\documentclass[reqno,final]{amsart}
\usepackage{fancyhdr} 
\usepackage{color} 
\usepackage{hyperref} 
\usepackage{graphicx} 
\usepackage{latexsym,amssymb,amsmath,amsfonts,amsthm}
\usepackage[notref,notcite]{showkeys}
\usepackage{float}
\usepackage{enumerate}
\usepackage{epic}
\usepackage{mathbbol}
\usepackage{fancyhdr}
\usepackage{graphicx}
\usepackage[toc,page]{appendix}
 \usepackage[mathscr]{euscript}
\usepackage{amsfonts}
\usepackage[all]{xy}
\usepackage{caption}
\definecolor{aleacolor}{rgb}{0,0,0}

\hypersetup{
breaklinks,
colorlinks=true,
linkcolor=aleacolor,
urlcolor=aleacolor,
citecolor=aleacolor}

\renewcommand{\headheight}{11pt}
\theoremstyle{plain}
\newtheorem{theorem}{Theorem}[section]                                          
                          
\newtheorem{lemma}[theorem]{Lemma}

\theoremstyle{definition}

\theoremstyle{remark}

\makeatletter \@addtoreset{equation}{section} \makeatother
\renewcommand\theequation{\thesection.\arabic{equation}}
\renewcommand\thefigure{\thesection.\arabic{figure}}
\renewcommand\thetable{\thesection.\arabic{table}}

\usepackage{epsfig}
\usepackage{subfigure}
\usepackage{amssymb}

\renewenvironment{proof}{\noindent{\textit{Proof:}}}{\hspace{2mm} $\square$}

\newcommand{\N}{\mathbf{N}}
\newcommand{\Z}{\mathbf{Z}}
\newcommand{\R}{\mathbf{R}}
\newcommand{\ind}{\mathbf{1}}
\newcommand{\ep}{\epsilon}

\newenvironment{demo}[1]{\noindent{\textit{Proof of #1.}}}{\hspace*{2mm}~$\square$}

\DeclareMathOperator{\card}{card \,}
\DeclareMathOperator{\bernoulli}{Bernoulli \,}
\DeclareMathOperator{\binomial}{Binomial \,}
\DeclareMathOperator{\exponential}{Exponential \,}

\begin{document}

\author{Eric Foxall}
\address{Department of Mathematics, University of Victoria, Victoria, BC V8X 1P3, Canada and School of Mathematical and Statistical Sciences, Arizona State University, Tempe, Arizona, AZ 85287-1804, USA.}
\email{eric.foxall@asu.edu}

\author{Nicolas Lanchier}
\address{School of Mathematical and Statistical Sciences, Arizona State University, Tempe, Arizona, AZ 85287-1804, USA.}
\email{nlanchie@asu.edu}


\title[Death and birth of the fittest]{Evolutionary games on the lattice: \\ death and birth of the fittest}

\begin{abstract}
 This paper investigates the long-term behavior of an interacting particle system of interest in the hot topic of evolutionary game theory.
 Each site of the~$d$-dimensional integer lattice is occupied by a player who is characterized by one of two possible strategies.
 Following the traditional modeling approach of spatial games, the configuration is turned into a payoff landscape that assigns a payoff to each player based on her strategy and the
 strategy of her neighbors.
 The payoff is then interpreted as a fitness, by having each player independently update their strategy at rate one by mimicking their neighbor with the largest payoff.
 With these rules, the mean-field approximation of the spatial game exhibits the same asymptotic behavior as the popular replicator equation.
 Except for a coexistence result that shows an agreement between the process and the mean-field model, our analysis reveals that the two models strongly disagree in many aspects,
 showing in particular that the presence of a spatial structure in the form of local interactions plays a key role.
 More precisely, in the parameter region where both strategies are evolutionary stable in the replicator equation, in the spatial model either one strategy wins or the system fixates
 to a configuration where both strategies are present.
 In addition, while defection is evolutionary stable for the prisoner's dilemma game in the replicator equation, space favors cooperation in our model.
\end{abstract}
 \subjclass[2000]{60K35, 91A22} 
\keywords{Interacting particle systems, evolutionary game theory, evolutionary stable strategy, replicator equation, prisoner's dilemma, cooperation.}
\maketitle


\section{Introduction}
\label{sec:intro}

\noindent There are a number of spatially explicit models of community dynamics, i.e., models where individuals are located on a geometrical
 structure and can only interact with nearby individuals, for which the long-term behavior has been proved to strongly differ from that of
 their non-spatial counterparts.
 See for instance~\cite{durrett_levin_1994, neuhauser_2001} for a discussion on this aspect.
 This shows that space is often a key component in how communities are shaped.
 This paper studies an interacting particle system of evolutionary games which turns out to be a good illustration of the fact that space matters
 in the sense that the results we collect show significant discrepancies between the spatially explicit stochastic model and its so-called
 mean-field version. \\
\indent Evolutionary game theory was proposed by theoretical biologist Maynard Smith and first appeared in~\cite{maynardsmith_price_1973}.
 Although evolutionary games were originally introduced to understand animal conflicts, they now have a number of important applications in
 various branches of social sciences.
 From a mathematical point of view, the stochastic process we consider is an example of interacting particle system (\cite{liggett_1985, liggett_1999})
 where each vertex of the~$d$-dimensional integer lattice is occupied by a player, who is characterized by her strategy.
 Thus, the state of the system at time~$t$ is a spatial configuration
 $$ \xi_t : \Z^d \to \{1, 2, \ldots, n \} = \hbox{set of strategies}. $$
 The dynamics depends on a payoff matrix~$A = (a_{ij})$ where $a_{ij}$ denotes the payoff of a type~$i$ player interacting with a type~$j$ player.
 In the presence of two strategies, which we assume from now on, there are exactly~$4! = 24$ possible orderings of the four payoffs.
 Since there is symmetry under exchange of type labels, we may assume~$a_{12} \le a_{21}$, which gives twelve possible games or strategic situations.
 The most popular such games are the prisoner's dilemma, the hawk-dove game and the stag hunt game.
 To define the evolution rules, we let~$M$ be an integer referred to as the range of the interactions and let
 $$ \begin{array}{rcl}
               N_x & = & \{z \neq x : \max_{i = 1, 2, \ldots, d} \,|z_i - x_i| \leq M \} \vspace*{4pt} \\
      f_i (x, \xi) & = & \card \{z \in N_x : \xi (z) = i \} / ((2M + 1)^d - 1) \end{array} $$
 be respectively the interaction neighborhood of site~$x$ and the fraction of neighbors of that site following strategy~$i$.
 Having a payoff matrix and a spatial configuration induces a so-called payoff landscape that assigns a payoff to each player on the infinite integer lattice by setting
 $$ \phi (x, \xi) = a_{i1} \,f_1 (x, \xi) + a_{i2} \,f_2 (x, \xi) \quad \hbox{when} \quad \xi (x) = i. $$
 The basic idea and starting point of evolutionary game theory is to turn the game into a dynamical system by interpreting the payoff as fitness or
 reproduction success, meaning that the type or strategy of a player is more likely to be selected at each interaction as her payoff gets larger.
 Specifically, the model considered in this paper evolves according to the following simple rules:
\begin{list}{\labelitemi}{\leftmargin=2.25em}
 \item[4.] Players update their strategy at rate one by mimicking their neighbor with the largest payoff.
           In case of a tie, meaning that two neighbors with different strategies share the same highest payoff, the new strategy is chosen uniformly at
           random by flipping a fair coin.
\end{list}
 This model is chronologically the last one of a series of four that we have studied so far.
 The first three models are built from the following rules.
\begin{list}{\labelitemi}{\leftmargin=2.25em}
 \item[1.] {\bf Payoff affecting birth and death rates}~\cite{lanchier_2015}.
                When a player has a positive payoff, at rate this payoff, one of her neighbors chosen at random adopts her strategy, whereas when her
                payoff is negative, at rate minus this payoff,
                she adopts the strategy of one of her neighbors chosen at random.
                This updating rule is inspired from~\cite{brown_hansell_1987}. \vspace*{2pt}
 \item[2.] {\bf Best-response dynamics}~\cite{evilsizor_lanchier_2014}.
                Players update their strategy at rate one in order to maximize their payoff, which depends on the strategy of their neighbors. \vspace*{2pt}
 \item[3.] {\bf Death-birth updating process}~\cite{evilsizor_lanchier_2016}.
                Players update their strategy at rate one by mimicking a neighbor chosen at random according to probabilities that are proportional to
                the neighbors' payoff.
                The weak selection approximation of this process, which is a voter model perturbation, has first been studied heuristically
                in~\cite{ohtsuki_al_2006} and then analytically in~\cite{chen_2013, cox_durrett_perkins_2013} using duality techniques for interacting particle systems.
\end{list}
 Note that the model we consider in this paper can be viewed as a nonlinear version of the death-birth updating process in the sense that, like in the death-birth
 process, players update their strategy at rate one by mimicking one of their neighbors, but unlike in the death-birth process, they always choose
 to imitate their fittest neighbor.
 In spite of their similarities, the two models exhibit different behaviors and the approaches to study them strongly differ.
 To turn our verbal description of the evolution rules into equations, we define
 $$ \Phi_i (x, \xi) = \sup \,\{\phi (z, \xi) : z \in N_x \ \hbox{and} \ \xi (z) = i \} \quad \hbox{for} \quad i = 1, 2, $$
 which by convention~$= - \infty$ when there is no type~$i$ neighbor.
 This random variable keeps track of the largest payoff among all the neighbors of site~$x$ who follow strategy~$i$.
 Then, the transition rates at vertex~$x$ are given by
\begin{equation}
\label{eq:ips-rates}
  \begin{array}{rcl}
    1 \ \to \ 2 & \ \hbox{at rate} \ & \ind \{\Phi_1 (x, \xi) < \Phi_2 (x, \xi) \} \vspace*{4pt} \\ && \qquad + \ (1/2) \,\ind \{\Phi_1 (x, \xi) = \Phi_2 (x, \xi) \} \vspace*{8pt} \\
    2 \ \to \ 1 & \ \hbox{at rate} \ & \ind \{\Phi_1 (x, \xi) > \Phi_2 (x, \xi) \} \vspace*{4pt} \\ && \qquad + \ (1/2) \,\ind \{\Phi_1 (x, \xi) = \Phi_2 (x, \xi) \}. \end{array}
\end{equation}
 In each transition, the first indicator function means that the player always mimics her fittest neighbor while the second indicator function means that,
 in case of a tie, she chooses her new strategy at random by tossing a fair coin.


\section{Main results}
\label{sec:results}

\noindent Since one of our main objectives is to understand and quantify the role of space in the interactions among players, the first step is to briefly study the mean-field model.
 Letting~$u_i$ denote the frequency of players following strategy~$i$ and assuming that the population is well-mixing gives the differential equation
 $$ \begin{array}{l}
     u_1' (t) = - \ u_1 \,\ind \{a_{11} u_1 + a_{12} u_2 > a_{21} u_1 + a_{22} u_2 \} \vspace*{4pt} \\ \hspace*{50pt}
                + \ u_2 \,\ind \{a_{11} u_1 + a_{12} u_2 < a_{21} u_1 + a_{22} u_2 \}. \end{array} $$
 To study the mean-field model and describe its bifurcation diagram, it is convenient to adopt the approach of~\cite{lanchier_2013} by introducing the new parameters
 $$ a_1 = a_{11} - a_{21} \quad \hbox{and} \quad a_2 = a_{22} - a_{12} $$
 and declaring strategy~$i$ to be
\begin{list}{\labelitemi}{\leftmargin=1.75em}
 \item {\bf altruistic} when~$a_i < 0$, i.e., a player with strategy~$i$ confers a lower payoff to a player with the same strategy than the other strategy, \vspace*{2pt}
 \item {\bf selfish} when~$a_i > 0$, i.e., a player with strategy~$i$ confers a higher payoff to a player with the same strategy than the other strategy.
\end{list}
 In terms of~$a_1$ and~$a_2$, the mean-field model reduces to
 $$ \frac{du_1}{dt} \ = \ \left\{\begin{array}{ccl} - u_1 & \hbox{when} & a_1 u_1 > a_2 u_2 \vspace*{2pt} \\
                                                        0 & \hbox{when} & a_1 u_1 = a_2 u_2 \vspace*{2pt} \\
                                                      u_2 & \hbox{when} & a_1 u_1 < a_2 u_2. \end{array} \right. $$
 Using that~$u_1 + u_2 = 1$ gives the three fixed points
 $$ e_1 = 1 \quad \hbox{and} \quad e_2 = 0 \quad \hbox{and} \quad e_* = \frac{a_2}{a_1 + a_2} $$
 whose global stability is as follows:
\begin{list}{\labelitemi}{\leftmargin=1.75em}
 \item strategy~$i$ wins if it is selfish and the other strategy altruistic:
       $$ \begin{array}{l} e_* \notin (0, 1) \quad \hbox{and} \quad \lim_{t \to \infty} u_i (t) = 1 \ \ \hbox{for all} \ \ u_i (0) \in (0, 1), \end{array} $$
 \item when both strategies are altruistic, coexistence occurs:
       $$ \begin{array}{l} e_* \in (0, 1) \quad \hbox{and} \quad \lim_{t \to \infty} u_1 (t) = e_* \ \ \hbox{for all} \ \ u_i (0) \in (0, 1), \end{array} $$
 \item when both strategies are selfish, the system is bistable:
       $$ \begin{array}{rcl}
            \lim_{t \to \infty} u_1 (t) = 0 & \hbox{for all} & u_1 (0) < e_* \in (0, 1) \vspace*{4pt} \\
            \lim_{t \to \infty} u_1 (t) = 1 & \hbox{for all} & u_1 (0) > e_* \in (0, 1). \end{array} $$
\end{list}
 Recalling that a strategy is said to be evolutionary stable if, when adopted by a population, it cannot be invaded by any alternative strategy starting at an
 infinitesimally small frequency, the dynamics of the mean-field model can be summarized as: a strategy is evolutionary stable if and only if it is selfish. \\
\begin{figure}[t]
 \centering
\scalebox{0.50}{\input{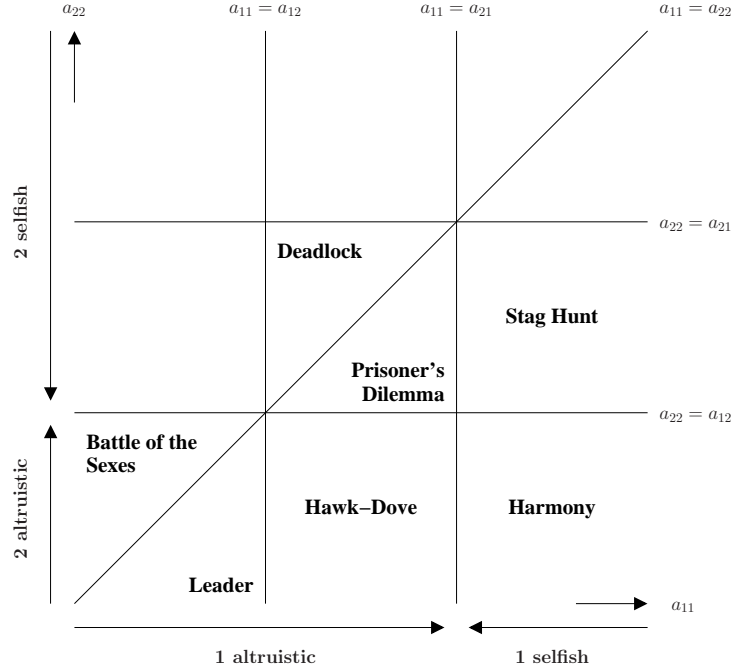}}
 \caption{\upshape List of the most popular two-person symmetric games.}
\label{fig:games}
\end{figure}
\indent We now turn our attention to the spatial model~\eqref{eq:ips-rates}.
 To visualize our results, we turn the four-dimensional parameter space into a two-dimensional picture by fixing the payoffs~$a_{12}$ and~$a_{21}$.
 The five straight lines
 $$ \{a_{11} = a_{12} \} \quad \{a_{11} = a_{21} \} \quad \{a_{22} = a_{12} \} \quad \{a_{22} = a_{21} \} \quad \{a_{11} = a_{22} \} $$
 induce a partition of the set of parameters into twelve regions corresponding to the twelve possible strategic relationships.
 These twelve regions are shown in Figure~\ref{fig:games} along with the names of the most popular games in the~$a_{11} - a_{22}$ plane under the assumption~$a_{12} < a_{21}$.
 Our main results stated and explained below are summarized in the two phase diagrams of Figures~\ref{fig:diagram-2D}--\ref{fig:diagram-1D}. \\
\indent To begin with, we look at parameter regions where at least one of the two strategies is selfish.
 To understand the spatial model in this case, we observe that, when the minimum payoff of a type~1 player exceeds the maximum payoff of a type~2 player over all possible configurations, i.e.,
\begin{equation}
\label{eq:bootstrap-1}
  \min \,(a_{11}, a_{12}) > \max \,(a_{21}, a_{22}),
\end{equation}
 the set of type~1 players becomes a pure growth process: type~1 players never change their strategy whereas each type~2 player with at least one type~1 neighbor changes her strategy at rate one.
 This implies that strategy~1 wins, in agreement with the mean-field model.
 This result can be improved upon by comparing the set of players for whom they and their neighbors eventually maintain strategy~1 to the limit of a certain bootstrap percolation process.
\begin{theorem}
\label{th:bootstrap}
 Assume that the process starts from a translation invariant product measure with a positive density of each type. Then,
\begin{enumerate}
 \item there is fixation to a configuration with a positive density of strategy~1 (in particular, strategy~2 cannot win) when
       $$ a_{11} > \max \,(a_{21}, a_{22}) $$
 \item strategy~1 wins when
       $$ (N - 1) \,a_{11} + \min \,(a_{11}, a_{12}) > N \max \,(a_{21}, a_{22}) $$
       where~$N = \card (N_x) = (2M + 1)^d - 1$ is the neighborhood size.
\end{enumerate}
\end{theorem}
\noindent By symmetry, we also have that
\begin{itemize}
 \item there is fixation to a configuration with a positive density of strategy~2 (in particular, strategy~1 cannot win) when
       $$ a_{22} > \max \,(a_{11}, a_{12}) $$
 \item strategy~2 wins when
       $$ (N - 1) \,a_{22} + \min \,(a_{21}, a_{22}) > N \max \,(a_{11}, a_{12}). $$
\end{itemize}
\begin{figure}[t]
 \centering
\scalebox{0.40}{\input{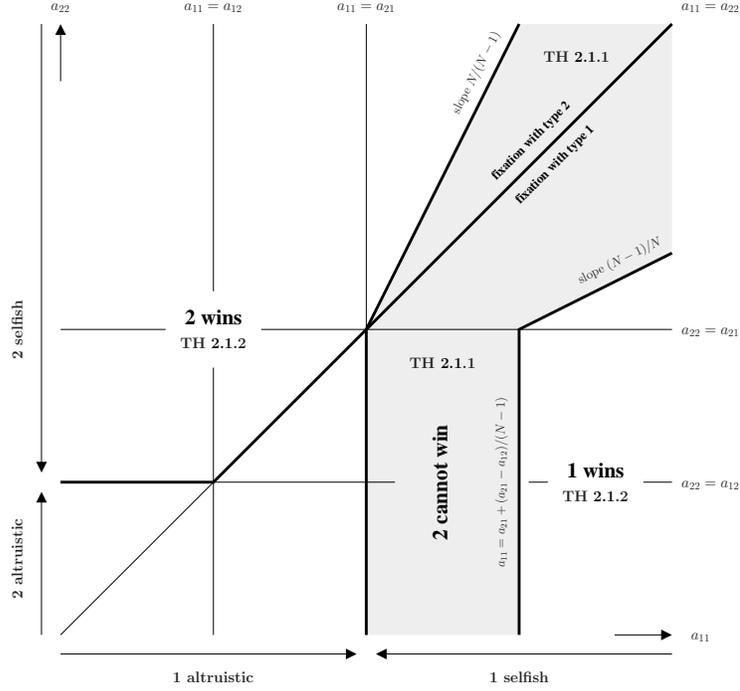}}
 \caption{\upshape Phase diagram of the death and birth of the fittest process in the~$a_{11} - a_{22}$ plane under the assumption that~$a_{12} < a_{21}$ showing the parameter regions covered
  by Theorem~\ref{th:bootstrap} delimited by the thick solid lines.}
\label{fig:diagram-2D}
\end{figure}
 See Figure~\ref{fig:diagram-2D} for a picture of the phase diagram in the~$a_{11} - a_{22}$ plane.
 Note that the parameter regions not covered by the theorem are
\begin{itemize}
 \item the region where both strategies are altruistic corresponding to the battle of the sexes, the leader game and the hawk-dove game, and \vspace*{2pt}
 \item the triangular region corresponding to the prisoner's dilemma game,
\end{itemize}
 and we now focus our attention on and collect results for the spatial model in these two more challenging parameter regions in the one-dimensional case. \\
\indent First, when both strategies are altruistic, numerical simulations suggest that coexistence occurs provided the range of the interactions is large.
 The result is difficult to prove in the general case but we were able to show it for the one-dimensional system under some symmetric condition.
\begin{theorem}
\label{th:long-range}
 Assume that
 $$ a_{11} = a_{22} < a_{12} = a_{21} \quad \hbox{and} \quad d = 1 $$
 and that the process starts from a translation invariant product measure with a positive density of each type. Then,
 $$ \begin{array}{l} \liminf_{t \to \infty} P (\xi_t (x) \neq \xi_t (y)) \neq 0 \quad \hbox{for all} \quad x \neq y \ \hbox{and} \ M \ \hbox{large enough}. \end{array} $$
\end{theorem}
\noindent To establish the theorem, we will show that if an interval of length~$7M$ has at least~$\log (M)$ players of each type then, with probability close to one,
 a certain translation of this interval will satisfy the same property after some fixed deterministic time.
 In addition to large deviation estimates to prove that the number of representatives of a given strategy in a given interval cannot vary too fast, the key ingredient is
 that, as long as the target interval does not contain any type~$i$ player, the rightmost type~$i$ player to the left of this interval is the fittest neighbor of all her
 neighbors with the possibility of a tie.
 This will be used to show that the minority type can spread and invade nearby regions. \\
\indent To understand the role of space in the prisoner's dilemma game, we now specialize in the one-dimensional system with nearest neighbor interactions.
 In this case, the process can be studied by looking at the interfaces between blocks of type~1 players and blocks of type~2 players. \\
\indent To begin with, we look at the system starting from an arbitrary configuration with a finite interval of players following strategy~1 not limited to translation invariant product measures.
 A natural question in this case is to determine whether the players in the initial interval of~1s may change their strategy or not.
 Our next result gives a lower bound for the probability that these players following initially strategy~1 never change their strategy under the assumption
\begin{equation}
\label{eq:1D-gold-1}
  a_{11} + a_{12} > a_{22} + \max \,(a_{21}, a_{22}) \quad \hbox{and} \quad 2a_{11} > a_{21} + a_{22},
\end{equation}
 This bound is uniform in all possible initial configurations outside the interval of type~1 players.
 More precisely, let
 $$ \Lambda_m = \{\xi \in \{1, 2 \}^{\Z} : \xi (x) = 1 \ \hbox{for all} \ x \in [-m, m] \cap \Z \} $$
 be the set of all the possible configurations that have a row of~$2m + 1$ players following strategy~1 centered at the origin, and let
 $$ \phi = \frac{1 + \sqrt 5}{2} = \hbox{the golden ratio}. $$
 Then, techniques from martingale theory give the following result.
\begin{theorem}
\label{th:1D-gold}
 Let~$M = d = 1$, $m \geq 1$ and assume~\eqref{eq:1D-gold-1}. Then,
 $$ \begin{array}{l} \inf_{\xi_0 \in \Lambda_m} P (\xi_t \in \Lambda_m \ \hbox{for all} \ t) \geq \displaystyle \bigg(1 - \frac{1}{\phi} \bigg)^2 = \frac{7 - 3 \sqrt 5}{2}. \end{array} $$
\end{theorem}
\noindent By symmetry, we have the same result for strategy~2 in the parameter region obtained from~\eqref{eq:1D-gold-1} by exchanging the role of the two strategies, i.e.,
\begin{equation}
\label{eq:1D-gold-2}
  a_{21} + a_{22} > a_{11} + \max \,(a_{12}, a_{11}) \quad \hbox{and} \quad 2a_{22} > a_{11} + a_{12}.
\end{equation}
 The two parameter regions~\eqref{eq:1D-gold-1}--\eqref{eq:1D-gold-2} covered by Theorem~\ref{th:1D-gold} can be conveniently visualized in the~$a_{11} - a_{22}$ plane as follows.
 Letting
 $$ \begin{array}{rcl}
      D_0 & = & \hbox{straight line with slope~1 going through} \ (a_{21}, a_{12}) \vspace*{2pt} \\
      D_1 & = & \hbox{straight line with slope~$\frac{1}{2}$ going through} \ (a_{12}, a_{12}) \vspace*{2pt} \\
      D_2 & = & \hbox{straight line with slope~2 going through} \ (a_{21}, a_{21}) \end{array} $$
 with the two payoffs~$a_{12}$ and~$a_{21}$ being fixed,
\begin{itemize}
 \item the payoffs~$a_{11}$ and~$a_{22}$ satisfy~\eqref{eq:1D-gold-1} if and only if point~$(a_{11}, a_{22})$ is below all three straight lines, \vspace*{2pt}
 \item the payoffs~$a_{11}$ and~$a_{22}$ satisfy~\eqref{eq:1D-gold-2} if and only if point~$(a_{11}, a_{22})$ is above all three straight lines.
\end{itemize}
 See Figure~\ref{fig:diagram-1D} for a picture of the three straight lines in the~$a_{11} - a_{22}$ plane. \\
\indent Finally, we assume that the process starts from a translation invariant product measure with a positive density of each strategy.
 Since in this case the initial configuration contains infinitely many rows of say three type~1 players, it directly follows from Theorem~\ref{th:1D-gold} and the ergodic theorem that,
 with probability one, at least one of these rows will expand without bound.
 This shows that strategy~1 wins when~\eqref{eq:1D-gold-1} holds and that strategy~2 wins when~\eqref{eq:1D-gold-2} holds.
 By identifying finite patterns with a mixture of 1s and 2s that are stable, i.e., patterns that cannot change under the dynamics, when
\begin{equation}
\label{eq:1D-fixation}
  \begin{array}{l}
    (2a_{11} > a_{21} + a_{22} \ \ \hbox{and} \ \ 2a_{22} > a_{11} + a_{12}) \quad \hbox{or} \vspace*{4pt} \\
    (2a_{11} > a_{21} + a_{22} > a_{11} + a_{12}) \quad \hbox{or} \quad (2a_{22} > a_{11} + a_{12} > a_{21} + a_{22}), \end{array}
\end{equation}
 since these patterns are finite and therefore occur infinitely often in the initial configuration, we can also prove that, when~\eqref{eq:1D-fixation} holds, the system fixates to a
 random configuration with infinitely many type~1 players and infinitely many type~2 players.
 Summarizing, we obtain the following theorem.
\begin{figure}[t]
 \centering
\scalebox{0.40}{\input{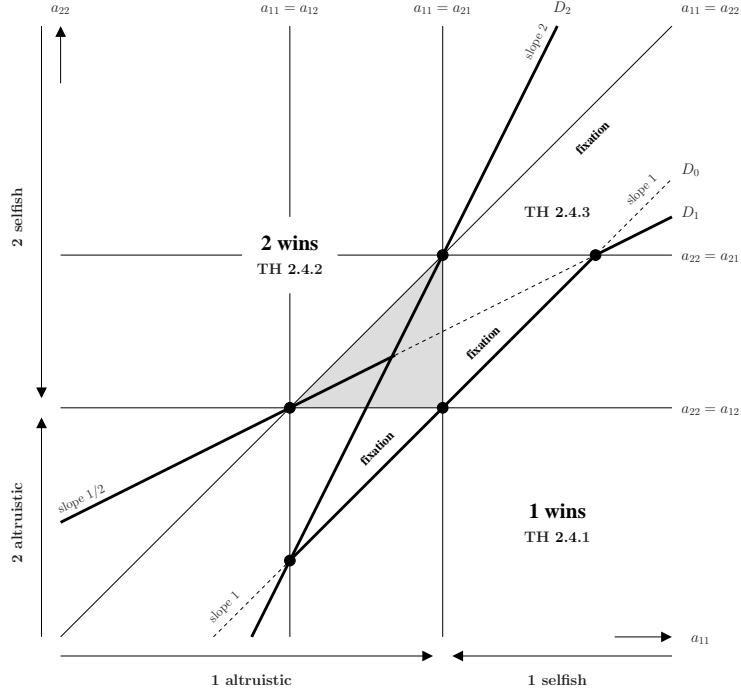}}
 \caption{\upshape Phase diagram of the one-dimensional nearest-neighbor process in the~$a_{11} - a_{22}$ plane under the assumption that~$a_{12} < a_{21}$.
  The parameter regions covered by Theorem~\ref{th:1D} are delimited by the thick solid lines.
  The grey triangle refers to the parameter region of the prisoner's dilemma game.}
\label{fig:diagram-1D}
\end{figure}
\begin{theorem}
\label{th:1D}
 Let~$M = d = 1$ and assume that the process starts from a translation invariant product measure with a positive density of each type. Then,
\begin{enumerate}
 \item strategy~1 wins when~\eqref{eq:1D-gold-1} holds, \vspace*{4pt}
 \item strategy~2 wins when~\eqref{eq:1D-gold-2} holds, \vspace*{4pt}
 \item there is fixation to a translation invariant random configuration in which both strategies are present when~\eqref{eq:1D-fixation} holds.
\end{enumerate}
\end{theorem}
\noindent Using again the three straight lines introduced above, the parameter region~\eqref{eq:1D-fixation} in which the process fixates corresponds to
\begin{itemize}
 \item all the parameter pairs~$(a_{11}, a_{22})$ which are below the straight line~$D_2$ and above either~$D_0$ or~$D_1$ when~$a_{12} < a_{21}$,
\end{itemize}
 and we refer to Figure~\ref{fig:diagram-1D} for a picture.
 Note that the figure shows several disagreements between the spatial and the mean-field models.
 For instance, in the prisoner's dilemma triangular region characterized by
 $$ a_{12} < a_{22} < a_{11} < a_{21} $$
 when strategy~1 is cooperation and strategy~2 is defection, defectors always win in the mean-field model whereas there is a subset of this parameter region in which there is fixation
 for the one-dimensional system.
 This subset is the triangle below the median starting at~$(a_{21}, a_{21})$ in the prisoner's dilemma triangle. \\
\indent The parameter region which is not covered by Theorem~\ref{th:1D} is the region below the straight line~$D_1$ but above~$D_2$, i.e.,
 $$ 2a_{11} < a_{21} + a_{22} \quad \hbox{and} \quad 2a_{22} < a_{11} + a_{12}. $$
 Note that~$D_0$ divides this region into two parts.
 To fix the ideas, we look at the sub-region below the straight line~$D_0$ which gives the condition
\begin{equation}
\label{eq:1D-open}
  2a_{11} < a_{21} + a_{22} \quad \hbox{and} \quad a_{11} + a_{12} > a_{22} + \max \,(a_{21}, a_{22}).
\end{equation}
 In this case,  starting with a finite interval of type~1 players while all the other players initially follow strategy~2, the length of the type~1 interval evolves according to a simple
 symmetric random walk except that, when the length reaches two, it cannot further shrink.
 In particular, due to recurrence, the system fluctuates, i.e., all the players change their strategy infinitely often with probability one.
 This implies that none of the two strategies wins and that Theorem~\ref{th:1D-gold} does not hold in the parameter region~\eqref{eq:1D-open}.
 Since in addition type~2 players with two type~1 neighbors change their strategy at rate one, when starting with a large but finite number of type~1 players, the type~1 intervals merge
 together until there is only one type~1 interval, which leads to the same conclusion: the system fluctuates.
 For the system starting with a translation invariant product measure, we conjecture that strategy~1 wins in the parameter region~\eqref{eq:1D-open} because the previous picture
 suggests that the connected component starting from an interval of type~2 players should be swallowed by the two nearby type~1 intervals.
 However, we have not been able to turn this heuristics into a proof because in the pattern
 $$ (2, 2, 2, 1, \underline{1}, 2, 1, 1, 1, 1, 1) $$
 the~1 which is underlined can turn into a~2, and so does the~1 next to it after such an update.
 This shows in particular that the length of each type~2 interval can make arbitrarily big jumps and is therefore difficult to control.


\section{Coupling with bootstrap percolation}
\label{sec:bootstrap}

\noindent This section is devoted to the proof of Theorem~\ref{th:bootstrap}.
 The first part of the theorem relies on a strong monotonicity result for the process rescaled in space that basically states that large hypercubes of type~1 players cannot break when
 \begin{equation}
 \label{eq:bootstrap-weak}
   a_{11} > \max \,(a_{21}, a_{22}).
 \end{equation}
 Under the stronger assumption of the second part
 \begin{equation}
 \label{eq:bootstrap-strong}
  (N - 1) \,a_{11} + \min \,(a_{11}, a_{12}) > N \max \,(a_{21}, a_{22}),
 \end{equation}
 we show that not only the rescaled process is almost surely nondecreasing but also it expands at least like a certain bootstrap percolation model which, together with a
 result in~\cite{schonmann_1992}, is the key to proving the second part of the theorem. \\
\indent Throughout this section, configurations of spin systems are often identified with the corresponding set of vertices in state~1 to simplify some of the expressions.
 In addition, we think of the evolutionary game model as being generated from a Harris' graphical representation~\cite{harris_1972}, i.e., a collection of independent Poisson processes.
 In contrast with the classical death-birth updating process in which players update their strategy by mimicking a neighbor chosen with a probability proportional to the neighbor's payoff,
 the new strategy is now chosen deterministically except when there is a tie.
 In particular, the graphical representation reduces to a collection of rate one Poisson processes, with one process attached to each vertex, together with a collection of independent
 Bernoulli random variables.
 More precisely, for each~$x \in \Z^d$,
\begin{itemize}
 \item we let $(N_t (x) : t \geq 0)$ be a rate one Poisson process, \vspace{3pt}
 \item we denote by $T_n (x)$ its $n$th arrival time: $T_n (x) = \inf \,\{t : N_t (x) = n \}$, \vspace{3pt}
 \item we let~$B_n (x) = \bernoulli (1/2)$ for~$n > 0$ be independent.
\end{itemize}
 We think of~$T_n (x)$ as the~$n$th time at which the player at~$x$ updates her strategy while the Bernoulli random variables are only used in case of a tie.
 Since the Poisson processes attached to different players are independent and the number of players is countable, with probability one, there are no simultaneous updates.
 Assuming that the process has been constructed up to time~$t-$ where~$t = T_n (x)$, the new configuration is determined from the following rule, where we abuse notation somewhat
 to view~$\xi$ as the set~$\{x : \xi (x) = 1 \}$.
\begin{itemize}
 \item In case~$x \in \xi_{t-}$, we set~$\xi_t = \xi_{t-} \setminus \{x \}$ if and only if
       $$ (\Phi_1 (x, \xi_t) < \Phi_2 (x, \xi_t)) \ \hbox{or} \ (\Phi_1 (x, \xi_t) = \Phi_2 (x, \xi_t) \ \hbox{and} \ B_n (x) = 0), $$
 \item while if~$x \notin \xi_{t-}$, we set~$\xi_t = \xi_{t-} \cup \{x \}$ if and only if
       $$ (\Phi_1 (x, \xi_t) > \Phi_2 (x, \xi_t)) \ \hbox{or} \ (\Phi_1 (x, \xi_t) = \Phi_2 (x, \xi_t) \ \hbox{and} \ B_n (x) = 1). $$
\end{itemize}
 The result of~\cite{harris_1972}, which holds for a large class of interacting systems, guarantees that our process can indeed be constructed using these rules starting from any initial
 configuration. \\
\indent We now focus on the proof of the theorem.
 Declare a player to be a 1-center when she and all her neighbors follow strategy~1.
 To keep track of such players, we introduce the process
 $$ \begin{array}{l} \eta_t (x) = \prod_{z \in \bar N_x} \ind \{\xi_t (z) = 1 \} \quad \hbox{where} \quad \bar N_x = N_x \cup \{x \}. \end{array} $$
 With this definition, the player at vertex~$x$ is a 1-center if and only if~$\eta_t (x) = 1$.
 Identifying every configuration~$\eta_t$ with the set of vertices in state~1, i.e., with the set of 1-centers, we have the following strong monotonicity result, which is the key to proving
 the first part of the theorem.
\begin{lemma}
\label{lem:bootstrap-weak}
 Assume~\eqref{eq:bootstrap-weak}.
 Then, $P (\eta_s \subset \eta_t) = 1$ for all~$s \leq t$.
\end{lemma}
\noindent In words, the result states that a 1-center is stable in the sense that it remains a 1-center forever.
 This is due to the fact that 1-centers get the largest possible payoff and therefore inspire all of their neighbors to stick to strategy~1.
 We now turn our heuristics into a rigorous proof. \\ \\
\begin{proof}
 Let~$x \in \eta_s$ be a 1-center at time~$s$.
 By definition of a 1-center, there is no player following strategy~2 in the neighborhood of vertex~$x$, therefore
\begin{equation}
\label{eq:bootstrap-weak-1}
  \begin{array}{l} \lim_{\ep \to 0} \,(1/\ep) \,P (x \notin \xi_{s + \ep} \,| \,x \in \eta_s) = 0. \end{array}
\end{equation}
 Observe in addition that, since~\eqref{eq:bootstrap-weak} holds, for all~$y \notin \xi_s$,
 $$ \begin{array}{rcl}
    \phi (y, \xi_s) & = & a_{21} \,f_1 (y, \xi_s) + a_{22} \,f_2 (y, \xi_s) \leq \max \,(a_{21}, a_{22}) \vspace*{4pt} \\
                    & < & a_{11} = a_{11} \,f_1 (x, \xi_s) + a_{12} \,f_2 (x, \xi_s) = \phi (x, \xi_s) \end{array} $$
 indicating that the payoff of a~1-center is always greater than the payoff of a type~2 player.
 This implies that, for all~$z \in N_x$ necessarily of type~1,
 $$ \begin{array}{rcl}
    \Phi_1 (z, \xi_s) & = & \sup_{w \in N_z} \phi (w, \xi_s) \,\ind \{w \in \xi_s \} \vspace*{4pt} \\
                   & \geq & \phi (x, \xi_s) > \sup_{w \in N_z} \phi (w, \xi_s) \,\ind \{w \notin \xi_s \} = \Phi_2 (z, \xi_s). \end{array} $$
 In particular, for all~$z \in N_x$,
\begin{equation}
\label{eq:bootstrap-weak-2}
  \begin{array}{l}
  \lim_{\ep \to 0} \,(1/\ep) \,P (z \notin \xi_{s + \ep} \,| \,x \in \eta_s) \vspace*{4pt} \\ \hspace*{10pt} =
  \ind \{\Phi_1 (x, \xi) < \Phi_2 (x, \xi) \} + (1/2) \,\ind \{\Phi_1 (x, \xi) = \Phi_2 (x, \xi) \} = 0. \end{array}
\end{equation}
 It follows from~\eqref{eq:bootstrap-weak-1}--\eqref{eq:bootstrap-weak-2} that, once a player and all of her neighbors are of type~1, which makes her a~1-center, they remain so forever.
\end{proof} \\ \\
 Recall that~$\bar N_x = N_x \cup \{x \}$.
 The first part of Theorem~\ref{th:bootstrap} directly follows from Lemma~\ref{lem:bootstrap-weak}.\\ \\
\begin{demo}{Theorem~\ref{th:bootstrap}.1}
 Assume that~$\rho > 0$ and~\eqref{eq:bootstrap-weak} holds. Then,
 $$ \begin{array}{l} P (x \in \eta_0) = P (\bar N_x \subset \xi_0) = \prod_{y \in \bar N_x} P (y \in \xi_0) = \rho^{N + 1} > 0 \end{array} $$
 where~$N = \card (N_x)$.
 This and Lemma~\ref{lem:bootstrap-weak} imply that
 $$ \begin{array}{l}
      \limsup_{t \to \infty} P (\xi_t (x) = 2) = 1 - \liminf_{t \to \infty} P (\xi_t (x) = 1) \vspace*{4pt} \\ \hspace*{25pt}
                                            \leq 1 - P (x \in \xi_t \ \hbox{for all} \ t) \leq 1 - P (x \in \eta_0) = 1 - \rho^{N + 1} < 1 \end{array} $$
 therefore strategy~2 cannot win.
\end{demo} \\ \\
 To prove the second part of the theorem, we compare the set of 1-centers with one of models of bootstrap percolation studied in~\cite{schonmann_1992}.
 This process is the discrete-time two-state spin system that starts from a random configuration but evolves deterministically as follows:
 occupied vertices remain occupied forever and empty vertices become occupied at the next time step if and only if at least one of their two nearest neighbors in each of the~$d$ directions is
 occupied.
 More formally, identifying the state of the process with the set of occupied sites, the dynamics of this bootstrap percolation model is described by
\begin{equation}
\label{eq:bootstrap}
  \begin{array}{l}
   x \in \zeta_{t + 1} \quad \hbox{if and only if} \vspace*{3pt} \\ \quad
   x \in \zeta_t \quad \hbox{or} \quad \zeta_t \cap \{x - e_i, x + e_i \} \neq \varnothing \quad \hbox{for all} \quad i = 1, 2, \ldots, d, \end{array}
\end{equation}
 for all~$(x, t) \in \Z^d \times \N$ where~$e_i$ is the~$i$th unit vector in~$\Z^d$.
 Since this process evolves in discrete time whereas our process evolves in continuous time, rather than a stochastic domination at every time, we will prove that the infinite time limit of
 bootstrap percolation is a subset of its counterpart for our process.
 To compare the infinite time limits, we first note that
 $$ \begin{array}{l}
      a_{11} \leq \max \,(a_{21}, a_{22}) \quad \hbox{implies that} \vspace*{4pt} \\ \hspace*{25pt}
     (N - 1) \,a_{11} + \min \,(a_{11}, a_{12}) \leq N a_{11} \leq N \max \,(a_{21}, a_{22}) \end{array} $$
 showing that condition~\eqref{eq:bootstrap-strong} implies condition~\eqref{eq:bootstrap-weak}.
 From this and Lemma~\ref{lem:bootstrap-weak}, we deduce that, whenever condition~\eqref{eq:bootstrap-strong} is satisfied, we have
\begin{equation}
\label{eq:boostrap-1}
  x \in \eta_s \quad \hbox{implies that} \quad x \in \eta_t \ \hbox{for all} \ t \geq s.
\end{equation}
 The second key ingredient is the following lemma.
\begin{lemma}
\label{lem:bootstrap-strong}
 Assume~\eqref{eq:bootstrap-strong}. Then, for all~$t \geq s$,
 $$ P (x \notin \eta_t \,| \,\eta_s \cap \{x - e_i, x + e_i \} \neq \varnothing \ \hbox{for all} \ i = 1, 2, \ldots, d) \leq e^{- (t - s)}. $$
\end{lemma}
\begin{proof}
 By symmetry of the evolution rules, it suffices to prove the upper bound for the conditional probability given the particular event
\begin{equation}
\label{eq:bootstrap-strong-1}
  A = \{x - e_i \in \eta_s \ \hbox{for all} \ i = 1, 2, \ldots, d \}.
\end{equation}
 Let~$x_+ = x + \sum_i M e_i$.
 Noticing that
 $$ \begin{array}{l} \bar N_x \cap \bar N_{x - e_i} = \{x + \sum_j c_j \,e_j : - M \leq c_j \leq M \ \hbox{and} \ c_i \neq M \} \end{array} $$
 for~$i = 1, 2, \ldots, d$, we deduce that
 $$ \begin{array}{l}
    \bigcup_i \,(\bar N_x \cap \bar N_{x - e_i}) = \bigcup_i \,\{x + \sum_j c_j \,e_j : - M \leq c_j \leq M \ \hbox{and} \ c_i \neq M \} \vspace*{4pt} \\ \hspace*{20pt}
                                                 = \{x + \sum_j c_j \,e_j : - M \leq c_j \leq M \ \hbox{and} \ \sum_j c_j \neq dM \} = \bar N_x \setminus \{x_+ \}. \end{array} $$
 Since~\eqref{eq:bootstrap-strong-1} and Lemma~\ref{lem:bootstrap-weak} imply that~$x - e_i \in \eta_t$ for all~$t \geq s$,
\begin{equation}
\label{eq:bootstrap-strong-2}
  \begin{array}{l} \bar N_x \setminus \{x_+ \} \subset \,\bigcup_i \bar N_{x - e_i} \subset \,\xi_t \end{array}
\end{equation}
 for all times~$t \geq s$,
 indicating that, except maybe for the player at vertex~$x_+$, the player at vertex~$x$ and all the players in the neighborhood of~$x$ follow strategy~1 at time~$s$ and after.
 Now, we distinguish two cases depending on whether the player at vertex~$x_+$ follows strategy~1 or strategy~2 at time~$s$. \vspace*{8pt} \\
{\bf Case 1.} Assume that~$x_+ \in \xi_s$.
 This and~\eqref{eq:bootstrap-strong-2} imply that~$\bar N_x \subset \xi_s$ so the player at~$x$ is a 1-center and it follows from Lemma~\ref{lem:bootstrap-weak} that
\begin{equation}
\label{eq:bootstrap-strong-3}
  x \in \eta_t \quad \hbox{and} \quad P (x \notin \eta_t \,| \,A \ \hbox{and} \ x_+ \in \xi_s) = 0 \quad \hbox{for all} \quad t \geq s.
\end{equation}
{\bf Case 2.} Assume now that~$x_+ \notin \xi_s$ and let
 $$ T = \inf \,\{t > s : t = T_n (x_+) \ \hbox{for some} \ n \in \N^* \} $$
 be the time of the next update at vertex~$x_+$.
 From~\eqref{eq:bootstrap-strong-2}, we get
 $$ \begin{array}{rcl}
    \phi (x, \xi_t) & = & a_{11} \,f_1 (x, \xi_t) + a_{12} \,f_2 (x, \xi_t) \vspace*{4pt} \\
                    & = & (1/N)(\card (N_x \setminus \{x_+ \}) \,a_{11} + a_{12}) \vspace*{4pt} \\
                    & = & (1/N)((N - 1) \,a_{11} + a_{12}) \vspace*{4pt} \\
                 & \geq & (1/N)((N - 1) \,a_{11} + \min \,(a_{11}, a_{12})) \end{array} $$
 for all times~$t \in (s, T)$.
 In particular, under the assumption~\eqref{eq:bootstrap-strong}, for all vertices~$y$ occupied by a type~2 player in the time interval~$t \in (s, T)$,
 $$ \begin{array}{rcl}
    \phi (x, \xi_t) & \geq & (1/N)((N - 1) \,a_{11} + \min \,(a_{11}, a_{12})) \vspace*{4pt} \\
                       & > &  \max \,(a_{21}, a_{22}) \geq \phi (y, \xi_t) \end{array} $$
 indicating that, just before the update time~$T$, the type~1 player at vertex~$x$ has a larger payoff than all type~2 players.
 Recalling the transition rates of the process and that vertex~$x$ is in the neighborhood of vertex~$x_+$, this implies that the state of~$x_+$ switches to~1 at time~$T$.
 In particular, using also the inclusion~\eqref{eq:bootstrap-strong-2} and Lemma~\ref{lem:bootstrap-weak}, we obtain the implications
 $$ \begin{array}{rcl}
       x_+ \in \xi_T & \hbox{implies that} & \bar N_x \subset \,\xi_T \vspace*{3pt} \\
                     & \hbox{implies that} & x \in \eta_T \vspace*{3pt} \\
                     & \hbox{implies that} & x \in \eta_t \ \ \hbox{for all} \ \ t \geq T. \end{array} $$
 Observing finally that~$T - s = \exponential (1)$, we get
\begin{equation}
\label{eq:bootstrap-strong-4}
 \begin{array}{l}
   P (x \notin \eta_t \,| \,A \ \hbox{and} \ x_+ \notin \xi_s) = P (T > t) \vspace*{4pt} \\ \hspace*{80pt} = P (T - s > t - s) = e^{- (t - s)} \end{array}
\end{equation}
 for all times~$t \geq s$. \vspace*{8pt} \\
 Combining~\eqref{eq:bootstrap-strong-3}--\eqref{eq:bootstrap-strong-4}, we conclude that
 $$ \begin{array}{rcl}
      P (x \notin \eta_t \,| \,A) & \leq & P (x \notin \eta_t \,| \,A \ \hbox{and} \ x_+ \in \xi_s) \vee P (x \notin \eta_t \,| \,A \ \hbox{and} \ x_+ \notin \xi_s) \vspace*{4pt} \\
                                       & = & \max \,(0, e^{- (t - s)}) = e^{- (t - s)} \end{array} $$
 which completes the proof.
\end{proof} \\ \\
 The second part of Theorem~\ref{th:bootstrap} follows from~\eqref{eq:boostrap-1} and Lemma~\ref{lem:bootstrap-strong} together with a result of~\cite{schonmann_1992} about the bootstrap
 percolation model. \\ \\
\begin{demo}{Theorem~\ref{th:bootstrap}.2}
 Summarizing~\eqref{eq:bootstrap-1} and the proof of Lemma~\ref{lem:bootstrap-strong}, the transition rates describing the evolution of the process~$(\eta_t)$, i.e., the dynamics of
 the~1-centers, are given by
 $$ \begin{array}{rcl}
      c_{0 \to 1} (x, \eta_t) & = & \ind \{\eta_t \cap \{x - e_i, x + e_i \} \neq \varnothing \ \hbox{for} \ i = 1, 2, \ldots, d \} \vspace*{4pt} \\
      c_{1 \to 0} (x, \eta_t) & = & 0. \end{array} $$
 Comparing with the evolution rules in~\eqref{eq:bootstrap}, this implies that
\begin{equation}
\label{eq:time-limit-1}
  \begin{array}{l} \lim_{t \to \infty} \zeta_t \subset \,\lim_{t \to \infty} \eta_t \quad \hbox{when} \quad \zeta_0 \subset \eta_0. \end{array}
\end{equation}
 Note that the infinite time limit indeed exists for both processes due to monotonicity, i.e., the set of vertices in state~1 cannot decrease.
 To also compare the initial configurations of the two processes, assume that, initially, vertices are independently in state~1 with probability
 $$ P (\xi_0 (x) = 1) = \rho \quad \hbox{and} \quad P (\zeta_0 (x) = 1) = \rho^{N + 1}. $$
 Then, for all~$F \subset \Z^d$ finite,
\begin{equation}
\label{eq:time-limit-2}
  \begin{array}{rcl} P (F \subset \eta_0) & = & P (F + \bar N_0 \subset \xi_0) \vspace*{4pt} \\
                                          & = & \rho^{\,\card (F + \bar N_0)} \geq (\rho^{N + 1})^{\card (F)} = P (F \subset \zeta_0) \end{array}
\end{equation}
 where we use the fact that
 $$ \begin{array}{rcl}
    \card (F + \bar N_0) & = & \card \{x + y : x \in F \ \hbox{and} \ y \in \bar N_0 \} \vspace*{4pt} \\
                      & \leq & \card (F) \,\card (\bar N_0). \end{array} $$
 Hence, the initial configuration of~1-centers dominates stochastically the initial configuration of occupied sites.
 Also, by~\cite[Theorem~3.1]{schonmann_1992}, when starting from a product measure with a positive density of occupied sites,
\begin{equation}
\label{eq:time-limit-3}
  \begin{array}{l} \lim_{t \to \infty} \zeta_t = \lim_{t \to \infty} \{x \in \Z^d : \zeta_t (x) = 1 \} = \Z^d. \end{array}
\end{equation}
 Combining~\eqref{eq:time-limit-1}--\eqref{eq:time-limit-3}, we conclude that, when~$\rho > 0$ and~\eqref{eq:bootstrap-strong} holds,
 $$ \begin{array}{l} \lim_{t \to \infty} \xi_t \supset \lim_{t \to \infty} \eta_t \supset \lim_{t \to \infty} \zeta_t = \Z^d \end{array} $$
 showing that strategy~1 wins.
\end{demo}

\section{Coexistence of altruistic strategies}
\label{sec:coexistence}

\noindent This section is devoted to the proof of Theorem~\ref{th:long-range}.
 In particular, it is assumed throughout this section that the spatial dimension~$d = 1$ and that
 $$ a_- = a_{11} = a_{22} \quad \hbox{and} \quad a_+ = a_{12} = a_{21} \quad \hbox{with} \quad a_- < a_+. $$
 The underlying ingredient is a block construction, which consists in coupling the process properly rescaled in space and time with oriented site percolation.
 To set the space scale, for all~$z \in \Z$, we consider the spatial regions
 $$ A_z = 7Mz + (- 3M/2, 3M/2)\quad \hbox{and} \quad B_z = 7Mz + (- 7M/2, 7M/2). $$
 Then, the main objective is to prove that there is a time~$T$ such that the following events occur with probability close to one when~$M$ is large.
\begin{enumerate}
 \item Starting with at least~$\log (M)$ players with strategy~1 in the interval~$B_0$, at least one of these players does not change her strategy by time~$T$. \vspace*{4pt}
 \item As long as the interval~$A_1$ does not contain any type~1 player, the rightmost such player to the left of this interval spreads strategy~1 to her right to bring a type~1 player in~$A_1$
       by time~$T/2$. \vspace*{4pt}
 \item Once there is a type~1 player in~$A_1$, strategy~1 spreads fast until there are at least~$\sqrt M$ players following strategy~1 in the interval~$B_1$. \vspace*{4pt}
 \item At least~$\log (M)$ of these~$\sqrt M$ players still follow strategy~1 at time~$T$ thus producing the same initial condition as in~1 but translated in space and time by the vector~$(7M, T)$.
\end{enumerate}
 In addition, the previous events, which are depicted in Figure~\ref{fig:block}, also occur with probability close to one after exchanging the role of both strategies.
 To turn our sketch into a proof, for every finite subset~$B \subset \Z$ and~$i = 1, 2$, we let
 $$ \Sigma_t^i (B) = \card \{x \in B : \xi_t (x) = i \} \quad \hbox{and} \quad \Sigma_t (B) = \Sigma_t^1 (B) \wedge \Sigma_t^2 (B) $$
 be the number of type~$i$ players in~$B$ and the number of players in~$B$ following the minority strategy, respectively.
 Also, let~$T (B)$ be the first time there is at least one type~1 and one type~2 player in~$B$, and consider the four good events
 $$ \begin{array}{rcl}
      D_1 & = & \Sigma_0 (B_0) > \log (M) \vspace*{4pt} \\
      D_2 & = & \Sigma_t (B_0) \neq 0 \ \hbox{for all} \ t \in (0, T) \vspace*{4pt} \\
      D_3 & = &  T (A_1) \leq T/2 \vspace*{4pt} \\
      D_4 & = & \Sigma_t (B_1) > \sqrt M \ \hbox{for some} \ t \in (0, T). \end{array} $$
 Note that these events correspond the the events described verbally in~1--4 above and are arranged in chronological order.
 The fact that each of these events occurs with probability close to one is proved in the following four lemmas again following their chronological order.
\begin{figure}[t]
 \centering
\scalebox{0.50}{\input{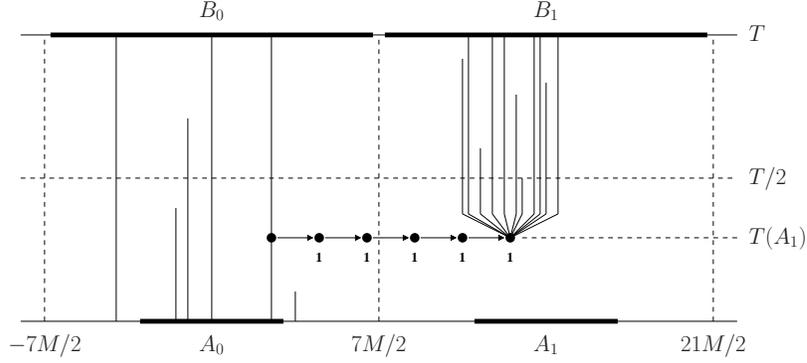}}
 \caption{\upshape Schematic picture of the four events in~1--4.}
\label{fig:block}
\end{figure}
\begin{lemma}
\label{lem:keep-one}
 For all time~$T > 0$,
 $$ P (\Sigma_t^1 (B_0) = 0 \ \hbox{for some} \ t \in (0, T) \,| \,D_1) \leq (1 - e^{-T})^{\log (M)}. $$
\end{lemma}
\begin{proof}
 On the event~$D_1$, there exists~$A \subset B_0$ such that
 $$ \card (A) = \log (M) \quad \hbox{and} \quad \xi_0 (x) = 1 \ \hbox{for all} \ x \in A. $$
 Then, the conditional probability to be estimated is bounded by the probability that all the players in~$A$ update their strategy by time~$T$.
 Since in addition players update their strategy independently at rate one, we deduce that
 $$ \begin{array}{l}
      P (\Sigma_t^1 (B_0) = 0 \ \hbox{for some} \ t \in (0, T) \,| \,D_1) \vspace*{4pt} \\
        \hspace*{15pt} \leq \ P (\Sigma_t^1 (A) = 0 \ \hbox{for some} \ t \in (0, T)) \leq P (T_1 (x) < T \ \hbox{for all} \ x \in A) \vspace*{4pt} \\
        \hspace*{15pt}   =  \ \prod_{x \in A} P (T_1 (x) < T) = (1 - e^{-T})^{\log (M)}. \end{array} $$
 This completes the proof.
\end{proof} \\ \\
 For every subset~$B \subset \Z$ and~$i = 1, 2$, define
 $$ T^i (B) = \inf \,\{t : \xi_t (x) = i \ \hbox{for some} \ x \in B \}. $$
\begin{lemma}
\label{lem:move}
 For all~$\ep > 0$, there exists~$T > 0$ such that
 $$ P (T^1 (A_1) > T/2 \,| \,D_1 \cap D_2) \leq \ep/8 \quad \hbox{for all~$M$}. $$
\end{lemma}
\begin{proof}
 Let~$X_t$ be the rightmost type~1 player to the left of~$7M - M/2$, i.e.,
 $$ X_t = \sup \,\{x < 7M - M/2 : \xi_t (x) = 1 \} $$
 and observe that, before time~$T^1 (A_1)$, the first~$3M$ sites to the right of~$X_t$ are occupied by type~2 players.
 It follows that, before the stopping time,
 $$ \begin{array}{rcl}
      2M - \Sigma_t (N_{X_t}) & \geq & M \geq \Sigma_t (N_{X_t + 1}) \vspace*{4pt} \\
                              & \geq & \Sigma_t (N_{X_t + 2}) \geq \cdots \geq \Sigma_t (N_{X_t + 2M}). \end{array} $$
 Using~$a_- < a_+$, this implies that
 $$ z \mapsto \phi_t (X_t + z, \xi_t) \ \hbox{is nonincreasing on} \ \{0, 1, \ldots, 2M \}, $$
 therefore, even in the worst case scenario of a tie,
\begin{equation}
\label{eq:move-1}
  \begin{array}{l} \lim_{s \to 0} \,s^{-1} \,P (X_{t + s} - X_t = M \,| \,\xi_t) \geq 1/2 \quad \hbox{almost surely}. \end{array}
\end{equation}
 Since jumps to the right are also bounded by the range~$M$, the process~$(X_t)$ cannot jump over~$A_1$ from left to right.
 In addition, in view of the size of the spatial regions, on the event that there is at least one type~1 player in~$B_0$ at any time before~$T$, the stopping time~$T^1 (A_1)$ is stochastically
 smaller than the time~$\tau_9$ it takes to see a sequence of nine jumps of length~$M$ to the right without jumps to the left.
 Note also that, in view of~\eqref{eq:move-1}, jumps of length~$M$ to the right occur at rate at least one half whereas, regardless of the payoffs, jumps to the left occur independently at rate one.
 This implies that time~$\tau_9$ is equal in distribution to the time until we see nine heads in a row while flipping a coin that comes up heads with probability one third at the
 times of a Poisson process with rate three halves.
 Clearly, time~$\tau_9$ is almost surely finite and its density function does not depend on the range of the interactions so there exists~$T$ large such that
 $$ P (T^1 (A_1) > T/2 \,| \,D_1 \cap D_2) \leq P (\tau_9 > T/2) \leq \ep/8 $$
 for all~$M$.
 This completes the proof.
\end{proof}
\begin{lemma}
\label{lem:spread}
 For all~$M$ large,
 $$ P (\Sigma_t^1 (B_1) \leq \sqrt M \ \hbox{for all} \ t \in (0, T) \,| \,D_1 \cap D_2 \cap D_3) \leq 1 / \sqrt M. $$
\end{lemma}
\begin{proof}
 To avoid trivialities, we assume that there are less than~$\sqrt M$ type~1 players in the interval~$B_1$ at time~$T (A_1)$.
 Then, on the event~$D_3$, there exists a site~$x_0 \in A_1$, fixed from now on, which is occupied by a type~1 player at time~$T (A_1)$, and we introduce the three stopping times
 $$ \begin{array}{rcl}
    \tau_{x_0} & = & \inf \,\{t > T (A_1) : \xi_t (x_0) = 2 \} \vspace*{4pt} \\
    \tau_-     & = & \inf \,\{t > T (A_1) : \Sigma_t^1 (B_1) > \sqrt M \} \vspace*{4pt} \\
    \tau_+     & = & \inf \,\{t > T (A_1) : \card \{x \in N_{x_0} : N_t (x) - N_{T (A_1)} (x) \neq 0 \} \geq \sqrt M \}. \end{array} $$
 Since~$a_- < a_+$, for all
 $$ T (A_1) < t < \min (\tau_{x_0}, \tau_-) \ \ \hbox{and} \ \ x \in x_0 + [- 2M, 2M] \ \ \hbox{with} \ \ \xi_t (x) = 2, $$
 we have the following lower and upper bounds for the payoffs:
 $$ \begin{array}{rcl}
    \phi (x_0, \xi_t) & = & a_{11} f_1 (x_0, \xi_t) + a_{12} f_2 (x_0, \xi_t) \vspace*{4pt} \\
                      & = & a_- f_1 (x_0, \xi_t) + a_+ f_2 (x_0, \xi_t) \vspace*{4pt} \\
                      & \geq & (2M)^{-1} \,(a_- \sqrt M + a_+ \,(2M - \sqrt M)), \ \hbox{and} \vspace*{8pt} \\
      \phi (x, \xi_t) & = & a_{21} f_1 (x, \xi_t) + a_{22} f_2 (x, \xi_t) \vspace*{4pt} \\
                      & = & a_+ f_1 (x, \xi_t) + a_- f_2 (x, \xi_t) \vspace*{4pt} \\
                      & \leq & (2M)^{-1} \,(a_+ \sqrt M + a_- \,(2M - \sqrt M)) < \phi(x_0, \xi_t) \end{array}$$
 for all~$M$ sufficiently large.
 In particular, for all~$z \in N_{x_0}$,
 $$ \begin{array}{rcl}
    \Phi_1 (z, \xi_t) & \geq & \phi (x_0, \xi_t) \vspace*{4pt} \\
                         & > & \sup_{x \in x_0 + [-2M, 2M]} \,\phi (x, \xi_t) \,\ind \{\xi_t (x) = 2 \} \geq \Phi_2 (z, \xi_t) \end{array} $$
 which implies that, at least between times~$T (A_1)$ and~$\min (\tau_{x_0},\tau_-)$, each update in the neighborhood of~$x_0$ results in strategy~1 being selected.
 It follows that
 $$ \{\tau_+ \leq \tau_{x_0} \} \subset \{\tau_- \leq \tau_+ \} \quad \hbox{almost surely} $$
 which, in turn, gives the inclusions of events
 $$ \begin{array}{rcl}
    \{\tau_- > T (A_1) + 1 \} & \subset & \{\tau_+ < \tau_- \} \cup \{\tau_+ > T (A_1) + 1 \} \vspace*{4pt} \\
                              & \subset & \{\tau_{x_0} < \tau_+ \} \cup \{\tau_+ > T (A_1) + 1 \}. \end{array} $$
 Taking the conditional probabilities on both sides, if~$T > 2$ we get
\begin{equation}
\label{eq:spread-1}
  \begin{array}{l}
    P (\Sigma_t^1 (B_1) \leq \sqrt M \ \hbox{for all} \ t \in (0, T) \,| \,D_1 \cap D_2 \cap D_3) \vspace*{4pt} \\
     \hspace*{25pt} \leq \ P (\Sigma_t^1 (B_1) \leq \sqrt M \ \hbox{for all} \ t \in (0, T (A_1) + 1] \,| \,D_1 \cap D_2 \cap D_3) \vspace*{4pt} \\
     \hspace*{25pt}   =  \ P (\tau_- > T (A_1) + 1 \,| \,D_1 \cap D_2 \cap D_3) \vspace*{4pt} \\
     \hspace*{25pt} \leq \ P (\tau_{x_0} < \tau_+ \ \hbox{or} \ \tau_+ > T (A_1) + 1) \vspace*{4pt} \\
     \hspace*{25pt} \leq \ P (\tau_{x_0} < \tau_+) + P (\tau_+ > T (A_1) + 1). \end{array}
\end{equation}
 Since the players in~$x_0 + [-M, M]$ update their strategy at the same rate, the permutations of the times of their first update after~$T (A_1)$ are equally likely, so
\begin{equation}
\label{eq:spread-2}
  \begin{array}{rcl}
    P (\tau_{x_0} < \tau_+) & \leq & P (\inf \,\{t > T (A_1) : N_t (x_0) - N_{t-} (x_0) = 1 \} < \tau_+) \vspace*{4pt} \\
                              &   =  & \sqrt M / (2M + 1) \leq 1 / (2 \sqrt M). \end{array}
\end{equation}
 In addition, letting~$X_i$ be independent exponential random variables with rate~$2M - i$ and using the memoryless property, we also have that
\begin{equation}
\label{eq:spread-3}
  \begin{array}{l}
    P (\tau_+ > T (A_1) + 1) = P (X_0 + X_1 + \cdots + X_{\sqrt M - 1} > 1) \vspace*{4pt} \\ \hspace*{50pt}
                          \leq P (X_i > 1 / \sqrt M \ \hbox{for some} \ i = 0, 1, \ldots, \sqrt M - 1) \vspace*{4pt} \\ \hspace*{50pt}
                          \leq \sqrt M \ P (X_M > 1 / \sqrt M) \vspace*{4pt} \\ \hspace*{50pt}
                             = \sqrt M \,\exp (- \sqrt M) \leq 1 / (2 \sqrt M). \end{array}
\end{equation}
 Combining~\eqref{eq:spread-1}--\eqref{eq:spread-3} gives the desired result.
\end{proof}
\begin{lemma}
\label{lem:keep-more}
 Let~$T > 0$. Then, for all~$M$ large,
 $$ \begin{array}{l}
      P (\Sigma_T^1 (B_1) \leq \log (M) \,| \,D_1 \cap D_2 \cap D_3 \cap D_4) \leq \exp (- (1/8) \,e^{-T} \,\sqrt M). \end{array} $$
\end{lemma}
\begin{proof}
 On the event~$D_4$, there exist~$A \subset B_1$ and~$t \in (0, T)$ such that
 $$ \card (A) = \sqrt M \quad \hbox{and} \quad \xi_t (x) = 1 \ \hbox{for all} \ x \in A. $$
 Since players update their strategy at rate one and the conditional probability of interest is bounded by the probability that (for~$M$ large) at least
 $$ \sqrt M - \log (M) \geq (1 - (1/2) e^{-T}) \,\sqrt M $$
 players in~$A$ update their strategy in less than~$T$ time units,
 $$ \begin{array}{l}
      P (\Sigma_T^1 (B_1) \leq \log (M) \,| \,D_1 \cap D_2 \cap D_3 \cap D_4) \vspace*{4pt} \\
             \hspace*{25pt} \leq \ P (\Sigma_t^1 (A) \leq \log (M) \ \hbox{for some} \ t \in (0, T) \,| \,\xi_0 (x) = 1 \ \hbox{for all} \ x \in A) \vspace*{4pt} \\
             \hspace*{25pt} \leq \ P (\card \{x \in A : T_1 (x) < T \} \geq (1 - (1/2) e^{-T}) \,\sqrt M) \vspace*{4pt} \\
             \hspace*{25pt} \leq \ P (\binomial (\sqrt M, 1 - e^{-T}) \geq (1 - (1/2) e^{-T}) \,\sqrt M). \end{array} $$
 Using also the standard large deviation estimate
 $$ P (\binomial (K, p) \geq K (p + z)) \leq \exp (- Kz^2 / 2 (1 - p)) \ \ \hbox{for} \ \ z \in (0, 1 - p) $$
 with~$(K, p) = (\sqrt M, 1 - e^{-T})$ and~$z = (1/2) e^{-T}$, we deduce that
 $$ \begin{array}{l}
      P (\Sigma_T^1 (B_1) \leq \log (M) \,| \,D_1 \cap D_2 \cap D_3 \cap D_4) \vspace*{4pt} \\
        \hspace*{25pt} \leq \ P (\binomial (\sqrt M, 1 - e^{-T}) \geq \sqrt M \,(1 - e^{-T} + (1/2) e^{-T})) \vspace*{4pt} \\
        \hspace*{25pt} \leq \ \exp (- \sqrt M \,(1/4) e^{-2T} / 2 e^{-T}) = \exp (- (1/8) \,e^{-T} \,\sqrt M) \end{array} $$
 for all~$M$ sufficiently large.
\end{proof} \\ \\
 To deduce coexistence from the previous four lemmas, we now couple the interacting particle system with one-dependent oriented site percolation in two dimensions in which sites are
 open with probability~$1 - \ep$.
 More precisely, we consider the directed graph with vertex set
 $$ \mathcal H = \{(z, n) \in \Z \times \Z_+ : z + n \ \hbox{is even} \} $$
 and in which there is an oriented edge
 $$ (z, n) \to (z', n') \quad \hbox{if and only if} \quad z' = z \pm 1 \ \hbox{and} \ n' = n + 1. $$
 Then, for the particle system, we declare~$(z, n) \in \mathcal H$ to be good if
 $$ \begin{array}{l} \Sigma_{nT} (B_z) = \min_{j = 1, 2} \,\card \{x \in B_z : \xi_{nT} (x) = j \} > \log (M), \end{array} $$
 indicating that there are at least~$\log (M)$ players of each strategy in~$B_z$ at time~$nT$, while oriented site percolation is defined by assuming that
 $$ P ((z_i, n_i) \ \hbox{is open for} \ i = 1, \ldots, m) = (1 - \ep)^m $$
 when~$|z_i - z_j| \vee |n_i - n_j| > 1$ for~$i \neq j$.
 Finally, for all~$n \in \N$, we let
 $$ W_n^{\ep} = \{z \in \Z : (z, n) \ \hbox{is wet} \} \quad \hbox{and} \quad X_n = \{z \in \Z : (z, n) \ \hbox{is good} \} $$
 where a site in~$\mathcal H$ is said to be wet whenever it can be reached from a directed path of open sites starting at level zero.
 The next lemma shows that, for all~$\ep > 0$, one can find a sufficiently large dispersal range such that the set of good sites dominates stochastically the set of wet sites.
 In view of the definition of a good site, this will imply coexistence of both strategies.
\begin{lemma}
\label{lem:perco}
 For all~$\ep > 0$, there is~$T > 0$ such that, for all~$M$ large,
 $$ P (z \in W_n^{\ep}) \leq P (z \in X_n) \quad \hbox{for all} \quad (z, n) \in \mathcal H \quad \hbox{whenever} \quad W_0^{\ep} \subset X_0. $$
\end{lemma}
\begin{proof}
 Fix~$T > 0$ large such that Lemma~\ref{lem:move} holds.
 Then, noting that~$D_1$ is the event that site~$(0, 0)$ is good and using the multiplication rule, we get
 $$ \begin{array}{l}
     P ((1, 1) \ \hbox{is good} \,| \,(0, 0) \ \hbox{is good}) \vspace*{4pt} \\
       \hspace*{20pt} \geq \ P (\Sigma_T (B_1) > \log (M) \ \hbox{and} \ D_2 \cap D_3 \cap D_4 \,| \,D_1) \vspace*{4pt} \\
       \hspace*{20pt}   =  \ P (D_2 \,| \,H_1) \,P (D_3 \,| \,H_2) \,P (D_4 \,| \,H_3) \,P (\Sigma_T^1 (B_1) > \log (M) \,| \,H_4) \end{array} $$
 where~$H_i = D_1 \cap \cdots \cap D_i$ for~$i = 1, 2, 3, 4$.
 Applying Lemmas~\ref{lem:keep-one}--\ref{lem:keep-more}, and observing that, since the condition on the payoffs is symmetric, each of these four lemmas also holds after exchanging the
 role of the two strategies, we deduce that
 $$ \begin{array}{l}
     P ((1, 1) \ \hbox{is good} \,| \,(0, 0) \ \hbox{is good}) \vspace*{4pt} \\
       \hspace*{20pt} \geq \ P (\Sigma_t (B_0) \neq 0 \ \hbox{for all} \ t \in (0, T) \,| \,H_1) \ P (T (A_1) \leq T/2 \,| \,H_2) \vspace*{4pt} \\
       \hspace*{40pt}        P (\Sigma_t (B_1) > \sqrt M \ \hbox{for some} \ t \in (0, T) \,| \,H_3) \vspace*{4pt} \\
       \hspace*{60pt}        P (\Sigma_T (B_1) > \log (M) \,| \,H_4) \vspace*{4pt} \\
       \hspace*{20pt} \geq \ (1 - 2 \,(1 - e^{-T})^{\log (M)})(1 - \ep/4) \vspace*{4pt} \\
       \hspace*{40pt}        (1 - 2 / \sqrt M)(1 - 2 \,\exp (- (1/8) \,e^{-T} \,\sqrt M)) \geq (1 - \ep/4)^4 \geq 1 - \ep \end{array} $$
 for all~$M$ large.
 By translation invariance of the graphical representation and invariance by spatial symmetry of the evolution rules, we have more generally
 $$ P ((z \pm 1, n + 1) \ \hbox{is good} \,| \,(z, n) \ \hbox{is good}) \geq 1 - \ep $$
 for all~$M$ large.
 Since also the proof of Lemmas~\ref{lem:keep-one}--\ref{lem:keep-more} shows that this lower bound holds uniformly in all the configurations outside the space-time region
 $$ (B_z \cup B_{z \pm 1}) \times [nT, (n + 1) T], $$
 we deduce from~\cite[Theorem 4.3]{durrett_1995} the existence of a coupling between the particle system and oriented site percolation such that
 $$ P (W_n^{\ep} \subset X_n) = 1 \quad \hbox{whenever} \quad W_0^{\ep} \subset X_0 $$
 from which the lemma follows directly.
\end{proof} \\ \\
 To deduce coexistence, we fix~$\ep > 0$ sufficiently small to make the oriented percolation process supercritical.
 Observing also that, starting from a product measure with a positive density of both strategies, the number of good sites at level zero is almost surely infinite, we obtain
\begin{equation}
\label{eq:coexistence-1}
  \begin{array}{l} \lim_{n \to \infty} P (0 \in W_{2n}^{\ep} \,| \,W_0^{\ep} = X_0) = \alpha > 0. \end{array}
\end{equation}
 The parameter~$\ep$ being fixed, we now fix~$T > 0$ as in Lemma~\ref{lem:perco} so that each wet site is also good for the coupling defined in the proof.
 Then, the definition of a good site implies that there exists a positive constant~$\beta$ such that
\begin{equation}
\label{eq:coexistence-2}
  P (\xi_{(2n + 1) T} (0) \neq \xi_{(2n + 1) T} (y - x) \,| \,0 \in W_{2n}^{\ep}) \geq \beta > 0
\end{equation}
 for all~$x \neq y$ and all~$M$ large.
 Combining~\eqref{eq:coexistence-1}--\eqref{eq:coexistence-2} and using again translation invariant in space, we conclude that
 $$ \begin{array}{l}
    \liminf_{t \to \infty} P (\xi_t (x) \neq \xi_t (y)) \vspace*{4pt} \\ \hspace*{20pt} = \
    \liminf_{n \to \infty} P (\xi_{(2n + 1) T} (0) \neq \xi_{(2n + 1) T} (y - x)) \vspace*{4pt} \\ \hspace*{20pt} \geq \
    \liminf_{n \to \infty} P (\xi_{(2n + 1) T} (0) \neq \xi_{(2n + 1) T} (y - x) \,| \,0 \in W_{2n}^{\ep}) \vspace*{4pt} \\ \hspace*{40pt} P (0 \in W_{2n}^{\ep} \,| \,W_0^{\ep} = X_0) \vspace*{4pt} \\ \hspace*{20pt} \geq \
    \alpha \,\beta > 0 \end{array} $$
 which completes the proof of the theorem.

\section{The one-dimensional nearest neighbor process}
\label{sec:1D}

\noindent This section is devoted to the proof of Theorems~\ref{th:1D-gold} and~\ref{th:1D}.
 In particular, we assume throughout this section that~$M = d = 1$. \\
\indent The proof of Theorem~\ref{th:1D-gold}, which focuses on the parameter region~\eqref{eq:1D-gold-1}, mainly relies on techniques from the theory of martingales.
 To begin with, we consider the process starting with only type~1 players to the left of and at the origin, i.e.,
 $$ \xi_0 \in \Lambda_- = \{\xi \in \{1, 2 \}^{\Z} : \xi (x) = 1 \ \hbox{for all} \ x \leq 0 \}. $$
 Our estimates hold uniformly in all possible~$\xi_0 \in \Lambda_-$. Let
 $$ X_t = \inf \,\{z \in \Z : \xi_t (z) = 2 \} - 1 \quad \hbox{and} \quad K_t = \inf \,\{z > 0  : \xi_t (X_t + z) = 1 \} $$
 be respectively the position of the rightmost type~1 player with only type~1 players to her left and the distance between this player and the next type~1 player to her right,
 as shown in the following picture.
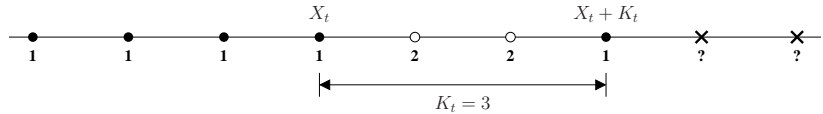
\begin{center}
\scalebox{0.50}{\input{half-line.pstex_t}}
\end{center}
 The first step of the proof is given by the following lemma which exhibits a supermartingale using the golden ratio~$\phi = (1/2)(1 + \sqrt 5)$.
\begin{figure}[t]
 \centering
 \scalebox{0.50}{\input{drift-1D.pstex_t}}
 \caption{\upshape Payoff landscape around~$X_t$ for different values of~$K_t$.}
\label{fig:drift-1D}
\end{figure}
\begin{lemma}
\label{lem:1D-super}
 Assume~\eqref{eq:1D-gold-1}.
 Then, $Z_t = \phi^{- X_t}$ is a supermartingale with respect to the natural filtration of the death and birth of the fittest process.
\end{lemma}
\begin{proof}
 For every set~$A$ and integer~$i$, we set
 $$ \begin{array}{l} p (\xi_t, A, i) = \lim_{s \to 0} s^{-1} \,P (X_{t + s} - X_t \in A \,| \,\xi_t \ \hbox{and} \ K_t = i). \end{array} $$
 Note that~$X_t$ can only move by one unit except when the distance~$K_t$ to the next player following strategy~1 is two, in which case it can make arbitrarily large jumps to the right.
 Using Figure~\ref{fig:drift-1D}, which shows the payoffs that are relevant to compute the rate of change of the process for different values of~$K_t$, we obtain
\begin{equation}
\label{eq:1D-super-1}
 \begin{array}{rcl}
   p (\xi_t, \{2, 3, \ldots \}, 2) & = & 1 \vspace*{4pt} \\
                  p (\xi_t, -1, 2) & = & \ind \{a_{11} < a_{21} \} + (1/2) \,\ind \{a_{11} = a_{21} \} \leq 1 \end{array}
\end{equation}
 with probability one.
 Similarly, since~\eqref{eq:1D-gold-1} holds,
\begin{equation}
\label{eq:1D-super-2}
  \begin{array}{l}
   p (\xi_t, +1, 3) = \ind \{a_{11} + a_{12} > a_{21} + a_{22} \} \vspace*{4pt} \\ \hspace*{80pt} + \ (1/2) \,\ind \{a_{11} + a_{12} = a_{21} + a_{22} \} = 1 \vspace*{4pt} \\
   p (\xi_t, -1, 3) = \ind \{2 a_{11} < a_{21} + a_{22} \} \vspace*{4pt} \\ \hspace*{80pt} + \ (1/2) \,\ind \{2 a_{11} = a_{21} + a_{22} \} = 0 \end{array}
\end{equation}
 with probability one, while for all~$i > 3$,
\begin{equation}
\label{eq:1D-super-3}
  \begin{array}{l}
   p (\xi_t, +1, i) = \ind \{a_{11} + a_{12} > 2 a_{22} \} \vspace*{4pt} \\ \hspace*{100pt} + \ (1/2) \,\ind \{a_{11} + a_{12} = 2 a_{22} \} = 1 \vspace*{4pt} \\
   p (\xi_t, -1, i) = \ind \{2 a_{11} < a_{21} + a_{22} \} \vspace*{4pt} \\ \hspace*{100pt} + \ (1/2) \,\ind \{2 a_{11} = a_{21} + a_{22} \} = 0 \end{array}
\end{equation}
 with probability one.
 From~\eqref{eq:1D-super-1}, we deduce that
 $$ \begin{array}{l}
    \lim_{s \to 0} \,s^{-1} \,E \,(Z_{t + s} - Z_t \,| \,\xi_t \ \hbox{and} \ K_t = 2) \vspace*{4pt} \\
    \hspace*{25pt} = \ \lim_{s \to 0} \,s^{-1} \,E \,(\phi^{- X_{t + s}} - \phi^{- X_t} \,| \,\xi_t \ \hbox{and} \ K_t = 2) \vspace*{4pt} \\
    \hspace*{25pt} \leq \  \phi^{- X_t} \,(\phi^{-2} - 1) \,\lim_{s \to 0} \,s^{-1} \,P (X_{t + s} \geq X_t + 2 \,| \,\xi_t \ \hbox{and} \ K_t = 2) \vspace*{4pt} \\
    \hspace*{50pt} + \ \phi^{- X_t} \,(\phi - 1) \,\lim_{s \to 0} \,s^{-1} \,P (X_{t + s} = X_t - 1 \,| \,\xi_t \ \hbox{and} \ K_t = 2) \vspace*{4pt} \\
    \hspace*{25pt} \leq \ Z_t \,((\phi^{-2} - 1) + (\phi - 1)) = Z_t \,(\phi^{-2} + \phi - 2) = 0. \end{array} $$
 Similarly, combining~\eqref{eq:1D-super-2}--\eqref{eq:1D-super-3}, we get that, for all~$i > 2$,
 $$ \begin{array}{l}
    \lim_{s \to 0} \,s^{-1} \,E \,(Z_{t + s} - Z_t \,| \,\xi_t \ \hbox{and} \ K_t = i) \vspace*{4pt} \\
    \hspace*{25pt} = \ \phi^{- X_t} \,(\phi^{-1} - 1) \,\lim_{s \to 0} \,s^{-1} \,P (X_{t + s} = X_t + 1 \,| \,\xi_t \ \hbox{and} \ K_t = i) \vspace*{4pt} \\
    \hspace*{50pt} + \ \phi^{- X_t} \,(\phi - 1) \,\lim_{s \to 0} \,s^{-1} \,P (X_{t + s} = X_t - 1 \,| \,\xi_t \ \hbox{and} \ K_t = i) \vspace*{4pt} \\
    \hspace*{25pt} = \ (\phi^{-1} - 1) \,Z_t \leq 0. \end{array} $$
 This shows that the process~$(Z_t)$ is a supermartingale.
\end{proof}
\begin{lemma}
\label{lem:stopping}
 Assume~\eqref{eq:1D-gold-1}. Then,
 $$ P (X_t \geq 0 \ \hbox{for all} \ t) \geq 1 - \phi^{-1} \quad \hbox{for all} \quad \xi_0 \in \Lambda_-. $$
\end{lemma}
\begin{proof}
 The idea is to apply the optional stopping theorem to~$(Z_t)$, the supermartingale found in the previous lemma, and to the stopping times
 $$ \tau_{-1} = \inf \,\{t : X_t = - 1 \} \quad \hbox{and} \quad \tau_n = \inf \,\{t : X_t \geq n \} \ \hbox{for all} \ n > 0. $$
 Note that~$T_n = \min \,(\tau_{-1}, \tau_n)$ is an almost surely finite stopping time because, regardless of the configuration of the system, the process~$(X_t)$ jumps to the right at rate one.
 Since in addition the initial position~$X_0 \geq 0$, the optional stopping theorem implies that
\begin{equation}
\label{eq:stopping-1}
  \begin{array}{rcl}
    1 & \geq & \phi^{- X_0} = E \,(Z_0) \geq E \,(Z_{T_n}) \vspace*{4pt} \\
      & \geq & E \,(Z_{T_n} \,| \,T_n = \tau_{-1}) \,P (T_n = \tau_{-1}) \vspace*{4pt} \\ && \hspace*{50pt} + \ E \,(Z_{T_n} \,| \,T_n = \tau_n) \,P (T_n = \tau_n) \vspace*{4pt} \\
      & \geq & \phi \,P (T_n = \tau_{-1}) = \phi \,(1 - P (T_n = \tau_n)). \end{array}
\end{equation}
 Now, using that the event~$\{T_n = \tau_n \}$ is nonincreasing in~$n$ for the inclusion, we also deduce from the monotone convergence theorem that
\begin{equation}
\label{eq:stopping-2}
  \begin{array}{rcl}
    P (X_t \geq 0 \ \hbox{for all} \ t) & \geq & P (X_t \geq 0 \ \hbox{for all} \ t \ \hbox{and} \ \lim_{t \to \infty} X_t = \infty) \vspace*{4pt} \\
                                          & \geq & P (T_n = \tau_n \ \hbox{for all} \ n > 0) \vspace*{4pt} \\
                                          &   =  & \lim_{n \to \infty} \,P (T_n = \tau_n). \end{array}
\end{equation}
 Combining~\eqref{eq:stopping-1}--\eqref{eq:stopping-2}, we conclude that, for all~$\xi_0 \in \Lambda_-$,
 $$ \begin{array}{rcl}
     P (X_t \geq 0 \ \hbox{for all} \ t) & \geq & \lim_{n \to \infty} \,P (T_n = \tau_n) \vspace*{4pt} \\
                                         & \geq & 1 - \phi^{-1} = (1/2)(3 - \sqrt 5), \end{array} $$
 which completes the proof.
\end{proof} \\ \\
 We now consider the process starting from a finite interval of type~1 players centered at the origin while the rest of the initial configuration is arbitrary.
 The basic idea to deduce Theorem~\ref{th:1D-gold} is to study the type~1 connected component starting at the origin, observing that the right and left boundaries of this component are described
 respectively by the process~$(X_t)$ and its mirror image. \\ \\
\begin{demo}{Theorem~\ref{th:1D-gold}}
 To begin with, we define formally the type~1 connected component starting at the origin as
 $$ C_1 (0) = \{(y, t) \in \Z \times \R_+ : (x, 0) \uparrow (y, t) \} $$
 where~$(x, 0) \uparrow (y, t)$ means that there exist
 $$ x = x_0, x_1, \ldots, x_n = y \quad \hbox{and} \quad 0 = t_0 \leq t_1 \leq \cdots \leq t_{n + 1} = t $$
 such that the following two conditions hold:
\begin{itemize}
 \item we have~$|x_i - x_{i + 1}| = 1$ for all~$i = 0, 1, \ldots, n - 1$ and \vspace*{4pt}
 \item we have~$\xi_s (z) = 1$ for all~$(z, s) \in \bigcup_{i = 0, 1, \ldots, n} (\{x_i \} \times [t_j, t_{j + 1}])$.
\end{itemize}
 In words, we can move forward in time from point~$(x, 0)$ to point~$(y, t)$ while staying in a space-time region occupied by type~1 players.
 Then, define the left and right boundaries of this connected component as
 $$ X_t^- = \inf \,\{z : (z, t) \in C_1 (0) \} \quad \hbox{and} \quad X_t^+ = \sup \,\{z : (z, t) \in C_1 (0) \} $$
 with the convention~$\inf (\varnothing) = \infty$ and~$\sup (\varnothing) = - \infty$.
 Then, it directly follows from Lemma~\ref{lem:stopping} that, for every integer~$m \geq 1$,
 $$ \begin{array}{l}
    \inf_{\xi_0 \in \Lambda_m} P (\xi_t \in \Lambda_m \ \hbox{for all} \ t) \vspace*{4pt} \\ \hspace*{15pt} = \
    \inf_{\xi_0 \in \Lambda_m} P (X_t^- \leq - m \ \hbox{and} \ X_t^+ \geq m \ \hbox{for all} \ t) \vspace*{4pt} \\ \hspace*{15pt} = \
    \inf_{\xi_0 \in \Lambda_-} P (X_t \geq 0 \ \hbox{for all} \ t) \ \inf_{- \xi_0 \in \Lambda_-} P (- X_t \leq 0 \ \hbox{for all} \ t) \vspace*{4pt} \\ \hspace*{15pt} = \
    (\inf_{\xi_0 \in \Lambda_-} P (X_t \geq 0 \ \hbox{for all} \ t))^2 \geq (1 - \phi^{-1})^2 = (1/2)(7 - 3 \sqrt 5). \end{array} $$
 This completes the proof of the theorem.
\end{demo} \\ \\
 As previously mentioned, when starting with a translation invariant product measure with a positive density of each of the two strategies, since the initial configuration contains infinitely
 many intervals with at least three type~1 players, it directly follows from Theorem~\ref{th:1D-gold} and the ergodic theorem that strategy~1 wins with probability one in the parameter
 region~\eqref{eq:1D-gold-1}.
 Using also obvious symmetry shows the first two parts of Theorem~\ref{th:1D}. \\
\indent We now focus on the last part of the theorem, which looks at the parameter region where the process fixates to a random configuration in which both strategies are present.
 The basic idea is to prove the existence of patterns with a mixture of both types of players that are stable under the dynamics of the process. \\ \\
\begin{demo}{Theorem~\ref{th:1D}.3}
 Since both the evolution rules of the process and the initial configuration are translation invariant, the function
 $$ u (t) = P (\xi_t (x) - \xi_t (x - 1) \neq 0) \quad \hbox{for all} \quad (x, t) \in \Z \times \R_+ $$
 does not depend on the choice of~$x$.
 The function~$u$ can be viewed as the density of interfaces.
 Because interfaces annihilate but do not give birth, the function~$u$ is non-increasing so it has a limit~$l$ as time goes to infinity.
 We start with some definitions and look at the three sub-regions in~\eqref{eq:1D-fixation} separately. \\
\indent By definition, a {\bf block} is an interval
 $$ \{x, x + 1, \dots, x + k \} \subset \Z $$
 and its {\bf length} is $k + 1$.
 We call this a {\bf block of~1s} for a configuration~$\xi$ if
 $$ \xi(x) = \xi (x + 1) = \cdots = \xi (x + k) = 1 $$
 and a {\bf maximal block of~1s} if in addition
 $$ \xi (x - 1) = \xi (x + k + 1) = 2. $$
 The analogous definition applies to~2s.
 A {\bf pattern} is a finite configuration
 $$ \sigma = (\sigma_0, \sigma_1, \dots, \sigma_k) \subset \{1, 2 \}^{k + 1} $$
 and we call this pattern {\bf stable} if
 $$ (\xi_t (x + i))_{i = 0}^k = \sigma \quad \hbox{implies} \quad (\xi_s (x + i))_{i = 0}^k = \sigma \ \ \hbox{for all} \ \ s > t.$$
{\bf Case 1.} Assume that
 $$ 2a_{11} > a_{21} + a_{22} \quad \hbox{and} \quad 2a_{22} > a_{11} + a_{12}. $$
 By symmetry, we may assume $a_{21}+a_{22} \ge a_{11}+a_{12}$, which combining with the above implies $a_{11}>a_{12}$. 
 We claim the patterns~$\sigma$ with
 $$ \sigma_i = 2 \ \ \hbox{for} \ \ i = 0, 1, \dots, k \ \ \hbox{and} \ \ k \geq 2 $$
 are stable.
 To see this, first note the transition rates at vertex~$x$ depend only on the states at that vertex and its four nearest neighbors.
 Let $I = \{x, \dots, x + k \}$ be a maximal block of~2s for~$\xi_t$ with~$k \geq 2$.
 Since~$\xi_t (x - 1) = 1$,
 $$ \phi (x - 1, \xi_t) \leq a_{12} + \max \,(a_{11}, a_{12}) \leq a_{11} + a_{12} < 2 a_{22} = \phi (x + 1, \xi_t) $$
 therefore~$\xi_t (x)$ has flip rate zero;
 by symmetry, the same holds for~$\xi_t (x + k)$ while this is trivially true for the other sites in the block because they are in the same state as their neighbors.
 The claim follows. \\
\indent Since the initial configuration~$\xi_0$ is distributed according to a product measure with a positive density of~2s, with probability one, the lattice~$\Z$ is partitioned by~$\xi_0$
 into a succession of maximal blocks of~2s, say~$(I_j)_{j \in \Z}$, of length at least three, separated by finite blocks~$(K_j)_{j \in \Z}$ of length at least one.
 Since the~$I_j$ are stable and longer than the range of dependence, the restriction of the process~$(\xi_t)$ on each of the finite blocks~$K_j$, independently for different~$j$, evolves like a
 finite state Markov chain on the state space of patterns with length~$\card (K_j) + 4$, with a fixed pair of type~2 sites at each boundary.
 Thus, fixation follows once we show that for any fixed finite length, such a chain is absorbing, and to do so we need to show that for any finite length pattern there is a sequence of flips,
 each with positive rate, that leads to an absorbing configuration.
 We do this by sequentially removing maximal blocks of certain types, until only stable types remain. \\
\indent For a maximal block of length one, both neighbors are of opposite type, which makes the flip rate positive.
 Thus, by flipping the type of all such blocks, we may assume that maximal blocks all have length at least two.
 Once this is done,
\begin{enumerate}
 \item maximal blocks of~1s of length two can be removed while \vspace*{2pt}
 \item maximal blocks of~1s of length at least three are stable
\end{enumerate}
 To see the first statement, note that in the pattern
 $$ \sigma = (2, 2, 1, 1, 2, 2), $$
 the central~1s both have positive flip rate since, under our assumption above, their type~1 neighbor has
 payoff~$a_{11} + a_{12} \leq a_{21} + a_{22}$, the payoff of their type~2 neighbor.
 For the second statement, it suffices to show the leftmost~1 in
 $$ \sigma = (2, 2, 1, 1, 1) $$
 has flip rate zero, which is clear as its type~2 neighbor has payoff~$a_{21} + a_{22} < 2 a_{11}$, the payoff of its type 1 neighbor.
 It remains to consider maximal blocks of length two, with type~2.
 In light of the above, we may assume a pattern of the form
 $$ \sigma = (1, 1, 1, 2, 2, 1, 1, 1), $$
 and that both triplets of~1s are stable.
 Then, by symmetry, each~2 has the same flip rate, which is either positive or zero.
 In the first case the pair of~2s can be removed, otherwise the pattern is stable.
 This completes the argument and proves fixation of the process. \\
\indent It remains to show that the limiting density of both~1s and~2s is positive.
 To do this, let~$\rho$ be the density of~1s in the initial product measure.
 Since blocks of~2s of length at least three are stable, for any~$x \in \Z$,
 $$ \begin{array}{l} \liminf_{t \to \infty} P (\xi_t (x) = 2) \geq P (\xi_0 (x + i) = 2, \ i = 0, 1, 2) = (1 - \rho)^3 > 0. \end{array} $$
 Looking now at type~1, since as shown above, blocks of~1s of length at least three have flip rate zero when the adjacent maximal blocks of~2s have length at least two, by sandwiching a block of~1s
 of length three between two blocks of~2s of length three to ensure stability, we find that
 $$ \begin{array}{l} \liminf_{t \to \infty} P (\xi_t (x) = 1) \geq \rho^3 \,(1 - \rho)^6 > 0. \end{array} $$
 This completes the proof in the first case. \vspace*{4pt} \\
 {\bf Case 2.} Assume that
 $$ 2a_{11} > a_{21} + a_{22} > a_{11} + a_{12}. $$
 If in addition~$2a_{22} > a_{11}+a_{12}$, this becomes a particular case of Case~1 that we have already treated above so we may assume that~$2a_{22} \leq a_{11} + a_{12}$.
 Together with the above, this implies that
 $$ a_{11} > a_{22} \quad \hbox{and} \quad a_{21} > a_{22}  \quad \hbox{and} \quad a_{21} > a_{12}. $$
 First, we claim that
\begin{equation}
\label{eq:claim-1}
  \begin{array}{l}
    (\xi_0 (x + i))_{i = 0}^2 \neq (1, 2, 1) \vspace*{4pt} \\ \qquad \hbox{implies that} \quad (\xi_t (x + i))_{i = 0}^2 \neq (1,2,1) \ \ \hbox{for all} \ \ t > 0. \end{array}
\end{equation}
 Since the transitions
 $$ (1, 1, 1) \to (1, 2, 1) \quad \hbox{and} \quad (2, 2, 2, 1) \to (2, 1, 2, 1) $$
 are not possible, and since the size of a maximal block decreases by at most one at any transition, it suffices to show the transition
 $$ (1, 2, 2, 1) \to (1, 1, 2, 1) $$
 cannot occur.
 This holds because if sites~$x$ and~$x + 1$ are in state~2 while sites~$x - 1$ and~$x + 2$ are in state~2 then
 $$ \phi (x - 1, \xi_t) \leq a_{12} + \max \,(a_{11}, a_{12}) < a_{21} + a_{22} = \phi (x + 1, \xi_t). $$
 In particular, the state at site~$x$ cannot flip.
 Symmetry implies that the same holds for site~$x + 1$ so our first claim~\eqref{eq:claim-1} follows. \\
\indent Next, we claim that for~$k \geq 6$,
\begin{equation}
\label{eq:claim-2}
  \begin{array}{l}
    (\xi_0 (x + i))_{i = 0}^k = (2, 2, 1, 1, \dots, 1, 2, 2) \vspace*{4pt} \\ \qquad \hbox{implies that} \quad (\xi_t (x + i))_{i = 2}^{k - 2} = (1, 1, \dots, 1) \ \ \hbox{for all} \ \ t > 0. \end{array}
\end{equation}
 It follows from~\eqref{eq:claim-1} and our assumption that
 $$ (\xi_t (x + i))_{i = 0}^2 \neq (1,2,1) \quad \hbox{for all} \quad t > 0. $$
 Assuming the conclusion in~\eqref{eq:claim-2} holds at time $t$,
 $$ (\xi_t (x), \xi_t (x + 1)) = (1, 1) \ \hbox{or} \ (2, 1) \ \hbox{or} \ (2, 2). $$
 The flip rate at~$x + 2$ is trivially zero in the first two cases.
 In the case of~$(2, 2)$, the payoff at site~$x + 1$ is~$a_{21} + a_{22} < 2 a_{11}$, the payoff at~$x+3$, so the flip rate is still zero.
 Symmetry completes the argument. \\
\indent As in Case 1, the lattice~$\Z$ is partitioned by~$\xi_0$ into blocks~$(I_j)_{j \in \Z}$ and~$(K_j)_{j \in \Z}$, where in this case, on each block~$I_j$, $\xi_0$ is of the form
 $$ (2, 2, 1, 1, \dots, 1, 2, 2) $$
 and each~$K_j$ is finite.
 We proceed as before to show the chain on~$K_j$ is absorbing.
 As before, we may assume there are no maximal blocks of length one.
 A maximal block of~2s of length greater than two may shrink, since in the pattern
 $$ (1, 1, 2, 2, \dots, 2), $$
 the leftmost~2 has a type~1 neighbor with payoff~$a_{11} + a_{12} \geq 2a_{22}$, the payoff of its type~2 neighbor.
 Thus we may assume all maximal blocks of~2s have length two.
 A maximal block of~1s of length two may shrink, as is easily checked from
 $$ (2, 2, 1, 1, 2, 2), $$
 so we may assume maximal blocks of~1s have length at least three.
 In light of the previous claims, a pattern whose maximal blocks of~2s have length at most two, and whose maximal blocks of~1s have length at least three, is stable, completing the argument. \\
\indent To conclude, it remains to show that the limiting density of each type is positive whenever~$\rho \in (0, 1)$.
 To deal with type~1, since the subset
 $$ (1,1,1) \ \ \hbox{in} \ \ (2, 2, 1, 1, 1, 2, 2) $$
 is stable, we find that
 $$ \begin{array}{l} \liminf_{t \to \infty} P (\xi_t (x) = 1) \geq \rho^3 \,(1 - \rho)^4 > 0. \end{array} $$
 On the other hand, since the subset
 $$ (1, 1, 1, 2, 2, 1, 1, 1) \ \ \hbox{in} \ \ (2, 2, 1, 1, 1, 2, 2, 1, 1, 1, 2, 2) $$
 is stable, we find that
 $$ \begin{array}{l} \liminf_{t \to \infty} P (\xi_t (x) = 2) \geq \rho^6 \,(1 - \rho)^6. \end{array} $$
 This completes the proof of the second case. \vspace*{4pt} \\
{\bf Case 3.} Assume that~$2a_{22} > a_{11} + a_{12} > a_{21} + a_{22}$.
 Fixation in this case can be deduced from the previous case using some obvious symmetry.
\end{demo} \\ \\
 The proof of Theorem~\ref{th:1D} is now complete. \\


\noindent\textbf{Acknowledgment}.
 Eric Foxall was supported in part by an NSERC PDF Award and Nicolas Lanchier by NSA Grant MPS-14-040958.



\begin{thebibliography}{100}

\bibitem{brown_hansell_1987}
 Brown, D. B. and Hansell, R. I. C. (1987). Convergence to an evolutionary stable strategy in the two-policy game. \emph{Am. Naturalist} \textbf{130} 929--940.

\bibitem{chen_2013}
 Chen, Y-T. (2013). Sharp benefit-to-cost rules for the evolution of cooperation on regular graphs. \emph{Ann. Appl. Probab.} \textbf{23} 637--664.

\bibitem{clifford_sudbury_1973}
 Clifford, P. and Sudbury, A. (1973). A model for spatial conflict. \emph{Biometrika} \textbf{60} 581--588.

\bibitem{cox_durrett_perkins_2013}
 Cox, J. T., Durrett, R. and Perkins, E. A. (2013). Voter model perturbations and reaction diffusion equations. \emph{Ast\'erisque} no. 349, vi+113 pp.

\bibitem{durrett_1995}
 Durrett, R. (1995). Ten lectures on particle systems. In \emph{Lectures on probability theory (Saint-Flour, 1993)}, volume 1608 of \emph{Lecture Notes in Math.}, pages 97--201. Springer, Berlin.

\bibitem{durrett_levin_1994}
 Durrett, R. and Levin, S. (1994). The Importance of Being Discrete (and Spatial). \emph{Theoretical Population Biology} \textbf{46} 363--394.

\bibitem{evilsizor_lanchier_2014}
 Evilsizor, S. and Lanchier, N. (2014). Evolutionary games on the lattice: best-response dynamics. \emph{Electron J. Probab.} \textbf{19}, no. 75, 12 pp.

\bibitem{evilsizor_lanchier_2016}
 Evilsizor, S. and Lanchier, N. (2016). Evolutionary games on the lattice: death-birth updating process. \emph{Electron J. Probab.} \textbf{21}, no. 17, 29 pp.

\bibitem{harris_1972}
 Harris, T. E. (1972). Nearest neighbor Markov interaction processes on multidimensional lattices. \emph{Adv. Math.} \textbf{9} 66--89.

\bibitem{holley_liggett_1975}
 Holley, R. A. and Liggett, T. M. (1975). Ergodic theorems for weakly interacting systems and the voter model. \emph{Ann. Probab.} \textbf{3} 643--663.

\bibitem{lanchier_2013}
 Lanchier, N. (2013). Stochastic spatial models of producer-consumer systems on the lattice. \emph{Adv. Appl. Probab.} \textbf{45} 1157--1181.

\bibitem{lanchier_2015}
 Lanchier, N. (2015). Evolutionary games on the lattice: payoffs affecting birth and death rates. \emph{Ann. Appl. Probab.} \textbf{25} 1108--1154.

\bibitem{liggett_1985}
 Liggett, T. M. (1985). Interacting particle systems. \emph{Grundlehren der Mathematischen Wissenschaften [Fundamental Principles of Mathematical Sciences]} \textbf{276} Springer-Verlag, New York, xv+488.

\bibitem{liggett_1999}
 Liggett, T. M. (1999). Stochastic interacting systems: contact, voter and exclusion processes. \emph{Grundlehren der Mathematischen Wissenschaften [Fundamental Principles of Mathematical Sciences]} \textbf{324} Springer-Verlag, Berlin, xii+332.

\bibitem{maynardsmith_price_1973}
 Maynard Smith, J. and Price, G. R. (1973). The logic of animal conflict. \emph{Nature} \textbf{246} 15--18.

\bibitem{neuhauser_2001}
 Neuhauser, C. (2001). Mathematical challenges in spatial ecology. \emph{Notices Amer. Math. Soc.} \textbf{48} 1304--1314.

\bibitem{ohtsuki_al_2006}
 Ohtsuki, H., Hauert, C., Lieberman, E. and Nowak, M. A. (2006). A simple rule for the evolution of cooperation on graphs and social networks. \emph{Nature} \textbf{441} 502--505.

\bibitem{schonmann_1992}
 Schonmann, R. (1992). On the behavior of some cellular automata related to bootstrap percolation. \emph{Ann. Probab.} \textbf{20} 174--193.

\end{thebibliography}

\end{document}